\documentclass{amsart}
\usepackage{graphicx}
\usepackage{amssymb}
\usepackage{amsfonts}
\usepackage{hyperref}
\usepackage{mathrsfs}
\usepackage[margin=2.8cm]{geometry}

  [1995/08/01 v0.1 Double stroke roman fonts]

\def\ds@whichfont{dsrom}
\relax

\DeclareMathAlphabet{\mathds}{U}{\ds@whichfont}{m}{n}
\DeclareMathOperator{\esssup}{esssup}

\swapnumbers
\newtheorem{theorem}{Theorem}[section]
\newtheorem{lemma}[theorem]{Lemma}
\newtheorem{corollary}[theorem]{Corollary}
\newtheorem{proposition}[theorem]{Proposition}
\theoremstyle{definition}
\newtheorem{definition}[theorem]{Definition}
\newtheorem{assumption}[theorem]{Assumption}
\newtheorem{remark}[theorem]{Remark}
\newtheorem{example}[theorem]{Example}

\numberwithin{equation}{section}
\theoremstyle{plain}

\numberwithin{equation}{section} 
\numberwithin{figure}{section} 
\theoremstyle{plain}
\theoremstyle{plain}
\theoremstyle{remark}
\newtheorem*{acknowledgement*}{Acknowledgement}

\newcommand{\cA}{{\mathcal A}}
\newcommand{\cB}{{\mathcal B}}
\newcommand{\cC}{{\mathcal C}}

\newcommand{\cE}{{\mathcal E}}
\newcommand{\cF}{{\mathcal F}}
\newcommand{\cG}{{\mathcal G}}
\newcommand{\cH}{{\mathcal H}}

\newcommand{\cJ}{{\mathcal J}}
\newcommand{\cK}{{\mathcal K}}
\newcommand{\cL}{{\mathcal L}}

\newcommand{\cS}{{\mathcal S}}

\newcommand{\cY}{{\mathcal Y}}
\newcommand{\cZ}{{\mathcal Z}}
\newcommand{\te}{{\theta}}

\newcommand{\Om}{{\Omega}}
\newcommand{\om}{{\omega}}
\newcommand{\ve}{{\varepsilon}}
\newcommand{\del}{{\delta}}
\newcommand{\Del}{{\Delta}}
\newcommand{\gam}{{\gamma}}
\newcommand{\Gam}{{\Gamma}}

\newcommand{\Sig}{{\Sigma}}
\newcommand{\sig}{{\sigma}}
\newcommand{\al}{{\alpha}}
\newcommand{\be}{{\beta}}

\newcommand{\la}{{\lambda}}

\newcommand{\bbC}{{\mathbb C}}
\newcommand{\bbE}{{\mathbb E}}

\newcommand{\bbN}{{\mathbb N}}
\newcommand{\bbP}{{\mathbb P}}
\newcommand{\bbR}{{\mathbb R}}

\newcommand{\bbZ}{{\mathbb Z}}
\newcommand{\bbI}{{\mathbb I}}

\begin{document}
\title[]{Large deviations, moment estimates and almost sure invariance principles for skew products with mixing base maps and expanding on the average fibers}
 \author{Yeor Hafouta \\
\vskip 0.1cm
Department  of Mathematics\\
The Ohio State University}
\email{yeor.hafouta@mail.huji.ac.il, hafuta.1@osu.edu}

\maketitle
\markboth{Y. Hafouta}{Spectral method}
\renewcommand{\theequation}{\arabic{section}.\arabic{equation}}
\pagenumbering{arabic}

\begin{abstract} 
In this paper we show how to apply classical probabilistic tools for  partial sums $\sum_{j=0}^{n-1}\varphi\circ\tau^j$ generated  by a skew product  $\tau$, built over a sufficiently well mixing base map and a random expanding dynamical system. Under certain regularity assumptions on the observable $\varphi$,
we obtain  a central limit theorem (CLT) with rates, a functional CLT, an almost sure invariance principle (ASIP), a moderate deviations principle, several exponential concentration inequalities  and Rosenthal type moment estimates for skew products with $\alpha, \phi$ or $\psi$ mixing base maps and expanding on the average random fiber maps.  All of the results are new even in the uniformly expanding case. The main novelty here (contrary to \cite{ANV}) is that the random maps are not independent, they do not preserve the same measure and the observable $\varphi$ depends also on the base space.
For stretched exponentially $\al$-mixing base maps our proofs
 are based on multiple correlation estimates, which make the classical method of cumulants applicable.
  For $\phi$ or $\psi$ mixing  base maps, we obtain an ASIP and maximal and concentration inequalities
 by establishing an $L^\infty$ convergence of the iterates $\cK^n$ of a certain transfer operator $\cK$ with respect to a certain sub-$\sig$-algebra, which yields  an appropriate (reverse) martingale-coboundary decomposition. 
\end{abstract}

\section{Introduction and a preview  of the main results}
\subsection{Quenched limit theorems for random random dynamical systems}
Let $(X,\cB,m)$ be a probability space and let $(\Om,\cF,\bbP,\sig)$ be an invertible ergodic probability preserving system. Let $T_\om:X\to X,\,\om\in\Om$ be a family of non-singular maps (i.e.\ $m\circ T_\omega^{-1}\ll m$) so that the corresponding skew product $\tau$ given by $\tau(\om,x)=(\sig\om,T_\om x)$ is measurable. A random dynamical system is formed by the sequence of compositions 
$$
T_\om^n x, n\geq0\, \text{ where }\,T_\om^n=T_{\sig^{n-1}\om}\circ\cdots\circ T_{\sig\om}\circ T_{\om}
$$
taken along the orbit of a ``random" point $\om$.  The system $(\Om,\cF,\bbP,\sig)$ is often referred to as the \textit{driving system}, and the map $\sig$ is often referred to as the \textit{base map}. 

Let $\varphi:\Om\times X\to\bbR$ be a measurable function (``an observable") and let $\mu$ be a $\tau$-invariant probability measure on $\Om\times X$. Then $\mu$ can be decomposed as $\mu=\int \mu_\om d\bbP(\om)$, where $\mu_\om$ is a family of probability measures on $X$ so that $(T_\om)_*\mu_\om=\mu_{\sig\om}$ for $\bbP$-a.e. $\om$.
Set $S_n\varphi=\sum_{j=0}^{n-1}\varphi\circ\tau^j$. Then 
$$S_n\varphi(\om,x):=S_n^\om\varphi(x)=\sum_{j=0}^{n-1}\varphi_{\sigma^j\om}\circ T_\om^j,$$ where $\varphi_\om(\cdot)=\varphi(\om,\cdot)$.
For $\bbP$ almost every $\om$ we can consider the sequence of functions $S_n^\om\varphi(\cdot)$ on the probability space $(X,\cB,\mu_\om)$ as random variables. Limit theorems for such sequences are called \textit{quenched limit theorems}. Among the first papers dealing with quenched limit theorems for random dynamical systems are \cite{Kifer-1996, Kifer-1998}, where in \cite{Kifer-1996} a quenched large deviations principle was obtained, and in \cite{Kifer-1998} a central limit theorem and a law of iterated logarithm were established.
Since then  quenched limit theorems for random dynamical systems were extensively studied. For instance, in \cite{DFGTV0,DH2, DH3, DHS} almost sure invariance principle (almost sure approximation by a sum of independent Guassians) was established for random expanding or hyperbolic maps $T_\om$, in \cite{HK, DH1} Berry-Esseen theorems (optimal rates in the CLT) were obtained for similar classes of maps and in \cite{HK, DFGTV1, DFGTV2, DS} local central limit theorems were achieved. In addition, in \cite{HafYT} several limit theorems were extended to random non-uniformly hyperbolic or expanding maps. We would also like to refer to \cite{ABR} for related results concerning mixing rates for random non-uniformly hyperbolic maps and to \cite{HNTV} for related results concerning sequential dynamical systems, were an almost sure invariance principle (ASIP) was obtained.
We note that in many of the examples these results are obtained for the unique measure $\mu$
 such that  $\mu_\om$ is absolutely continuous with respect to $m$. However, some results hold true even for maps $T_\om:\cE_\om\to\cE_{\sigma\om}\subset X$ which are defined on random subsets of $X$ (see \cite{KifRanDy}), where in this case the most notable choices of $\mu_\om$ are the, so called, random Gibbs measures (see \cite{HK, MSU}). 
 

\subsection{Limit theorems skew-products}\label{skew}
Let us consider the sums $S_n\varphi=\sum_{j=0}^{n-1}\varphi\circ\tau^j$ as random variables on the probability space  $(\Om\times X,\cF\times\cB,\mu)$.  In this paper will focus on limit theorems for such sequences of random variables.
In order to demonstrate the difference between  such limit theorems and the quenched ones, let us focus of the CLT. The quenched CLT means that 
for $\bbP$-a.e. $\om$, for all real $t$ we have 
$$
\lim_{n\to\infty}\mu_\om\left(\left\{x: S_n^\om\varphi(x)-\mu_\om(S_n^\om\varphi)\leq t\sqrt n\right\}\right)=
\frac{1}{\sqrt{2\pi}\sig}\int_{-\infty}^{t}e^{-\frac{s^2}{2\sig^2}}ds
$$
where $\sig\geq0$ is the number which satisfies that  $\sig^2=\lim_{n\to\infty}\frac1n\text{Var}_{\mu_\om}(S_n^\om\varphi)$ for $\bbP$-a.e. $\om$ (see \cite{Kifer-1998}).
On the other hand, the CLT for the skew product means that 
 for all real $t$ we have 
$$
\lim_{n\to\infty}\mu\left(\left\{x: S_n\varphi(x)-\mu(S_n\varphi)\leq t\sqrt n\right\}\right)=
\frac{1}{\sqrt{2\pi}\Sig}\int_{-\infty}^{t}e^{-\frac{s^2}{2\Sig^2}}ds
$$
where  $\Sig^2=\lim_{n\to\infty}\frac1n\text{Var}_{\mu}(S_n\varphi)$.
Note that, in contrast with  the quenched case, the summands $X_j=\varphi\circ\tau^j$ form a stationary sequence and, in applications,  the existence of the limit $\Sig^2$ follows from a sufficiently fast decay of $\text{Cov}(X_0,X_n)$ as $n\to\infty$. We also remark that both CLT's above are formulated when $\sig$ and $\Sig$ are positive, and when one of them vanishes then the convergence is towards the constant function $0$.

When $\mu_\om(\varphi_\om)$ does not depend on $\om$ then $\mu_\om(\varphi_\om)=\mu(\varphi)$ and $\sig^2=\Sig^2$. In this case  the quenched CLT implies the CLT for $S_n\varphi$ by integrating   $\mu_\om\left(\left\{x: S_n^\om\varphi(x)-\mu_\om(S_n^\om\varphi)\leq t\sqrt n\right\}\right)$ with respect to $\bbP$ (and similarly other distributive limit theorems for the skew product follow from the quenched ones). 
However, it is less likely to be true when $\mu_\om(\varphi_\om)$ depends on $\om$. Remark that even when  $\mu_\om(\varphi_\om)$ does not depend on $\om$ other finer results like the ASIP do not follow by integration. Indeed the ASIP concerns an almost sure approximation of the partial sums at question by a sum of independent Guassian random variables, but the quenched ASIP provides a construction of such a Guassian process which depends on the fiber $\om$.

\subsubsection{\textbf{Annealed limit theorems: iid maps}}
A particular well studied case is when the maps $T_{\sig^j\om}$ are independent. That is, 
 $\Om=\cY^\bbZ$ is a product space, the coordinates $\om_j$ of $\om=(\om_j)$ are independent (with $\sig$ being the left shift) and $T_\om=T_{\om_0}$ depends only on the $0$-th coordinate. In this case the statistical behavior of the skew product $\tau$ can be investigated using the, so called, annealed transfer operator, given by (see \cite{BalYo, Bal, Ish}),
$$
\cA g(x)=\int \cL_\om g(x)d\bbP(\om)
$$ 
where $\cL_\om$ is the transfer operator corresponding to $T_\om$ and the underlying reference measure $m$. 
In \cite{ANV} it was shown that for several classes of random expanding maps, the operator $\cA$ is quasi compact. 
Using that, a variety of limit theorems were obtained\footnote{Such as a central limit theorem, a Berry-Esseen theorem, a local central limit theorem, local large deviations principle and an almost sure invariance principle.} for random variables of the form 
$$
S_n\varphi(\om,x)=\sum_{j=0}^{n-1}\varphi(T_{\om_{j-1}}\circ\dots\circ T_{\om_0}x)
$$ 
where $(\om,x)$ are distributed according  to a $\tau$-invariant measure $\mu$ of the form $\bbP\times (h\,dm)$ for some continuous function $h$, which satisfies $\cA h=h$. The latter assumption means that the maps $T_\om$ preserve the same measure $\nu=h\, dm$.
 The point is that once quasi compactness is achieved the classical Nagaev-Guivarch method (see \cite{HH}) can be applied. This method was applied successfully to obtain limit theorems 
for deterministic dynamical systems, i.e. when $T_\om=T$  does not depend on $\om$, and  in \cite{ANV} (see also \cite{Ayvar}) this method was applied to obtain annealed limit theorems.
We note that since both the function $\varphi$ and  the measure $h \,dm$ do not depend on $\om$, and all the maps $T_\om$ preserve the measure $h \,dm$ the fiberwise centering constant $\mu_\om(S_n^\om\varphi)$ and the usual centering constant $\mu(S_n\varphi)$ are both equal to $n\int \varphi(x)h(x)dm$. Hence, as discussed in the previous section, in this setup some annealed results like the CLT already follow from the quenched ones.
 
Independence here is crucial, since it yields that the iterates on the annealed transfer operator can be written as 
\begin{equation}\label{Ann it}
\cA^n g=\int \cL_{\om}^ng\,d\bbP(\om),
\end{equation}
where $\cL_\om^n=\cL_{\sig^{n-1}\om}\circ\dots\circ\cL_{\sig\om}\circ\cL_\om$, which is the transfer operator of $T_\om^n$. Hence, the statistical behavior of the iterates $\tau^n$ of the skew product can be described by the iterates of $\cA$. Note that in this iid setup this approach works only when $\varphi(\om,x)=\varphi(x)$ does not depend on $\om$ since it requires substituting $\varphi$ (and appropriate functions of $\varphi$) into the annealed operator.

\subsubsection{\textbf{The motivation behind the present paper: non iid maps and random functions}}
The starting point of this paper is the observation that  when the coordinates $(\om_j)$ are not independent\footnote{That is, that maps $T_{\sig^j\om}$ are not iid.} there is no apparent relation between the iterates $\tau^n$ of $\tau$ and the iterates of the annealed operator $\cA$ defined above.
Thus, a natural question arising from \cite{ANV, Ayvar} is which limit theorems hold true for mixing base maps with non-independent coordinates, and functions $\varphi$ which depend on $\om$. Moreover, the assumptions in \cite{ANV} require all the maps $T_\om$ to preserve the same absolutely continuous measure $\nu=h\, dm$, and it is also desirable to prove limit theorems without such assumptions.  We note that without the above assumptions even the CLT was  not obtained before for the skew products considered in this paper, which will be our first result.

The question described above was  also one of the main motivations in \cite{ETDS}, where  a CLT, a local CLT and a renewal theorem were obtained for several classes of skew products with mixing base maps such as Markov shifts and non-uniform Young towers, together with uniformly expanding random maps. These results were obtained by a certain type of integration argument, however the method of \cite{ETDS} does not involve the iterates of an annealed transfer operator, and instead we studied directly  integrals of the form $\int \cL_{\om}^ng_\om d\bbP(\om)$, and their complex perturbations (relying on the fiberwise ``spectral" properties  and a certain type of periodic point approach which was  introduced in \cite{HK}). While \cite{ETDS} was the first paper to discuss limit theorem for skew products with non independent fiber maps and random observables, all the results there were obtained for fiberwise centered observables $\varphi$ (i.e. $\mu_\om(\varphi_\om)=0$). Moreover, the maps $T_\om$ in \cite{ETDS} were uniformly expanding, the base map had a periodic point  and the random transfer operator satisfied certain regularity assumptions as functions of $\om$ around the periodic orbit. 
From this point of view, a second motivation for the present paper is to prove limit theorem for skew products with non-independent fiber maps $T_{\sig^j\om}$ without the fiberwise centralization assumption and without additional topological assumptions like the behavior around a periodic orbit. We note that apart from the CLT we did not consider in \cite{ETDS} any of the limit theorems obtained in the present paper, and so almost all the results in the present paper are new even under the fiberwise centering assumption.

\subsection{Our new results and the method of the proofs}\label{New}
As explained in the previous section,
the goal of this paper is to obtain limit theorems with deterministic centering conditions for skew products $\tau$ built over mixing base maps and non-uniformly expanding maps $T_\om$. More precisely, we still consider a product space $\Om=\cY^\bbZ$, but with ``weakly-dependent" coordinates $\om_j$ instead of independent ones. We consider a family of non-uniformly expanding map $T_\om=T_{\om_0}$ and observables of the form $\varphi(\om,x)=\varphi_{\om_0}(x)$ and  prove limit theorems for sequences of the form $Z_n=S_n\varphi-n\int \varphi d\mu$, where
$$
S_n\varphi(\om,x)=\sum_{j=0}^{n-1}\varphi_{\om_j}(T_{\om_{j-1}}\circ\dots\circ T_{\om_0}x)=\sum_{j=0}^{n-1}\varphi_{\sigma^j\om}(T_{\om}^j(x))
$$
considered as a random variables on the probability space $(\Om\times X,\cF\times\cB,\mu)$, where 
$\mu$ is the unique $\tau$-invariant measure with $\mu_\om$ being absolutely continuous with respect to $m$ (or when $\mu_\om$ is a random Gibbs measure).
 These results are obtained for a certain type of observables $\varphi$ so that $\varphi_{\om}(\cdot)$ has bounded variation, uniformly in $\om$. When the maps $T_\om$ are expanding on the average  we will  also have 
  a certain scaling assumption\footnote{That is $\esssup_{\om\in\Om}(K(\om)\|\varphi_\om\|_{BV})<\infty$ for some tempered random variable $K$.}, which was shown in \cite{DHS} to be necessary  for quenched limit theorems, and which is similarly necessary for obtaining limit theorems for the skew product. In what follows we will always assume that $\int\varphi d\mu=0$, which is not really a restriction since we can always replace $\varphi$ with $\varphi-\int\varphi d\mu$.

We obtain our results using two different methods, as described below.
\subsubsection{\textbf{A (functional) CLT,  moment estimates, moderate deviations and exponential concentration inequalities for $\al$-mixing driving systems via the method of cumulants}}\label{SA}
We assume first that the coordinates $\om_n$ are $\al$-mixing, with the $n$-th $\alpha$ mixing coefficient $\al_n$ (defined in \eqref{al def}) satisfy $\al_n=O(e^{-cn^\eta})$ for some $c,\eta>0$ (i.e. it is stretched exponential). The first step towards limit theorems  is standard for stationary processes: we  show that under the weaker condition $\sum_{n}n\al_n<\infty$, the limit 
$$
s^2=\lim_{n\to\infty}\frac 1n\text{Var}_\mu(S_n),\,S_n=S_n\varphi
$$ 
exists and that it vanishes if and only if $\varphi$ admits a certain co-boundary representation. When $s^2>0$ we show that $n^{-1/2}S_n$ converges in distribution towards a centered normal random variable with variance $s^2$. More precisely, we obtain the convergence rate 
$$
\sup_{t\in\bbR}\left|\mu(S_n\leq ts\sqrt n)-\frac{1}{\sqrt{2\pi}}\int_{-\infty}^t e^{-\frac12x^2}dx\right|\leq Cn^{-\frac1{2+4\gamma}},\,\gamma=1/\eta.
$$
An annealed CLT (that is for independent maps) was obtained in \cite{Ayvar} for random toral automorphism and in \cite{ANV} for more general maps. When the base map is only mixing (and $\varphi$ depend on $\om$) it was obtained in \cite{ETDS} for fiberwise centered potentials (i.e. $\mu_\om(\varphi_\om)=0$).  One of the results in this paper is the CLT for stretched exponentially $\al$-mixing base maps but without the fiberwise centering assumption (in fact, we will obtain a functional CLT, see Theorem \ref{FCLT} and the last paragraph of this section).

We also obtain certain type of large deviations results, which is often refereed to as a moderate deviations principle (see \cite{DemZet}). These results yield, for instance, that for every closed interval $[a,b]$ we have 
$$
\lim_{n\to\infty}\frac{1}{a_n^2}\ln\mu\left\{(\om,x):\,\frac{S_n(\om,x)}{a_nsn^{1/2}}\in[a,b]\right\}=-\frac12\inf_{x\in[a,b]}x^2
$$
where $a_n$ is a sequence so that $a_n\to\infty$ and $a_n=o(n^{\frac{1}{2+4\gam}})$. We also obtain several types of ``stretched" exponential concentration inequalities \eqref{FirstExpCon}, \eqref{ModDev3} and Gaussian moment estimates of Rosenthal type \eqref{Rosen}. These result are obtained using the method of cumulants. More precisely, we first obtain a certain type of multiple correlation estimates (see Proposition \ref{PropMulti}),  and then by applying general theorem  we conclude  that the $k$-th cumulant of the sum $S_n$ is at most of order $ n(k!)^{1+\gam}(c_0)^{k-2}$ for $k\geq 3$, where $c_0$ is some constant (see Theorem \ref{CumEst}). Then we can apply the method of cumulants \cite{Stat, DE}. In the annealed setup, using the quasi compactness of the annealed transfer operator large deviations principles and exponential concentration inequalities were obtained in \cite{ANV}, and the above results show that there is a similar behavior when the maps are not independent and the function $\varphi$ depends on $\om$ (see also the results in the next section where better exponential concentration inequalities are described).

The above multiple correlation estimates together with the method of cumulants and the Rosenthal type moment estimates also yield a functional CLT. Let us consider the random function $\cS_n(t)=n^{-1/2}S_{[nt]}$ on $[0,1]$. Then we show that it converges in distribution in the Skorokhod space $D[0,1]$ to $sW$, where $W$ is a standard Brownian motion and
$s^2=\lim_{n\to\infty}\frac 1n\text{Var}_\mu(S_n)$.

\subsubsection{\textbf{Limit theorems $\phi$ or $\psi$ mixing driving systems via martingale methods: almost sure invariance principle, concentration inequalities and maximal moment estimates}}\label{SM}
One of the strongest methods to prove central limit theorems and related results in probability theory and dynamical systems is the, so called, martingale-coboundary representation (Gordin's method). For a sufficiently chaotic dynamical system $(Y,\cG,\mu,T)$ and an observable $\varphi:Y\to\bbR$ it means that $\varphi$ can be represented as $\varphi=u+\chi-\chi\circ T$ for some sufficiently regular function $\chi$, and $(u\circ T^n)$ forms a reverse martingale difference. Such results are well known for deterministic  expanding (or hyperbolic) dynamical systems, and we refer to \cite{DFGTV0, Zemer, DHS} for a quenched and sequential versions of such martingale methods.
When the base map $\sig$ of the skew product $\tau$ is either (sufficiently fast) $\phi$ or $\psi$ mixing (see \eqref{phi def} and \eqref{psi def} for the relevant definitions) we obtain a certain type of $L^\infty$ martingale-coboundary  representation (i.e. $\chi\in L^\infty$) for the underlying class of observables $\varphi$ with respect to the skew product $\tau$. This was already established in \cite{ANV} in the annealed setup, and here using different arguments we obtain such a representation for skew products with mixing base maps.


Once an $L^\infty$ martingale-coboundary decomposition is achieved, as usual, we can apply the Azuma-Hoeffding inequality together with  Chernoff's bounding method and obtain exponential concentration inequalities of the form
$$
\bbP(|S_n-\bbE[S_n]|\geq tn+c_1)\leq c_2e^{-c_3 nt^2}, t>0
$$
where $c_1,c_2,c_3$ are positive constants. These concentration inequities 
are better
 than the ones we obtain using the method of cumulants, although they involve  the stronger notions of $\phi$ or $\psi$ mixing instead of $\al$-mixing\footnote{However, they only require summable $\phi$ or $\psi$ mixing coefficients and not stretched exponential ones.}.
 Another immediate consequence is moment estimates of the form
 $$
\left\|\max_{1\leq k\leq n}|S_k-\bbE[S_k]|\right\|_{L_p}=O(n^{1/2})
$$
 which hold for every $p\geq1$. Such results are known in the annealed case \cite{ANV}, and  we extend them to the skew products considered in this paper.

The idea behind the martingale-coboundary representation is as follows. 
Consider the sub-$\sig$-algebra $\cF_0$ of $\Om\times X$  generated by the projection $\pi_0(\om,x)=((\om_j)_{j\geq0},x)$, where $\om=(\om_j)_{j\in\bbZ}$. Then $\tau$ preserves $\cF_0$ since $T_\om=T_{\om_0}$ depends only on $\om_0$, and $\cF_0$ can be viewed as a sub-system (or a factor) given by $(\Om\times X,\cF_0,\mu,\tau)$. Our main argument is that, under quite mild $\phi$ or $\psi$ mixing rates for the coordinates $\om_j$, the iterates $\cK^n \varphi$  of the transfer operator $\cK$ corresponding to this system converge fast enough  in $L^\infty(\mu)$ towards $\mu(\varphi)\textbf{1}$, where $\textbf{1}$ is the function taking the constant value $1$, and $\varphi$ is our given observable. This convergence can be established for every function $\varphi$ so that $\|\varphi\|_{K,2}=\text{esssup}_{\om\in\Om}\left(K(\om)^2\|\varphi(\om,\cdot)\|_{BV}\right)<\infty$ for an appropriate tempered random variable $K(\om)$, or for any observable with $\esssup_{\om\in\Om}\|\varphi(\om,\cdot)\|_{BV}<\infty$ when the maps $T_\om$ are uniformly expanding. We stress that in any case this is not a spectral result (even under exponential mixing), since the convergence of $\cK^n$ is not in an operator norm, and, in general, it does not have exponential rate. Indeed we only prove that 
\begin{equation}\label{1}
\left\|\cK^n\varphi-\mu(\varphi)\right\|_{L^\infty}\leq C\|\varphi\|_{K,2}\cdot\gamma_n
\end{equation}
where $\gamma_n=\del^n+\phi_R([n/2])$ or $\gamma_n=\del^n+\psi([n/2])$, and $\del\in(0,1)$ and $\phi_R(\cdot)$ and $\psi(\cdot)$ are the reverse $\phi$-mixing coefficients and $\psi$-mixing coefficients defined in \eqref{phi def} and \eqref{psi def}, respectively.

Another consequence of the martingale-coboundary representation is the almost sure invariance principle (ASIP). In \cite{CM} the authors proved that under certain assumptions, a reverse martingale $M_n$ can be approximated almost surely by a sum of independent Guassians. One consequence of the methods in \cite{CM} is for sums of the form $W_n=\sum_{j=0}^{n-1}\varphi\circ \tau^j$. For such sums, the conditions of \cite[Theorem 3.2]{CM}
 shows that there is  a coupling with a sequence of iid centered normal random variables $Z_j$ with variance $s^2=\lim_{n\to\infty}\frac1n\text{Var}(W_n)$ so that 
$$
\sup_{1\leq k\leq n}\left|W_k-\sum_{j=1}^k Z_j\right|=O(n^{1/4}(\log n)^{1/2}(\log\log n)^{1/4})\,,\text{   almost surely}.
$$
In our notations, the first and second conditions of \cite[Theorem 3.2]{CM} about $\cK$ can be  verified using \eqref{1}.
In order to show that the third (and last condition) about $\cK$ in \cite[Theorem 3.2]{CM} is in force we will also need to  provide more general estimates on expression of the form
$$
\left\|\cK^i(\bar\varphi \cK^j\bar\varphi)-\mu\big(\cK^i(\bar\varphi \cK^j\bar\varphi)\big)\right\|_{L^\infty}
$$
for $1\leq i,j\leq n$, where $\bar\varphi=\varphi-\mu(\varphi)$.

We note that in \cite{ANV} the annealed ASIP was obtained using Gouezel's approach \cite{GO} and not the martingale-coboundary approach. Gouezel's approach was also used in \cite{Atnip} to obtain an ASIP for non-independent maps with mixing base maps, but as indicated in \cite{Atnip} the results are mostly applicable for Gordin-Denker maps.

Finally, we  also prove a vector-valued almost sure invariance principle for skew products with uniformly expanding random maps and exponentially fast $\al$-mixing base maps  via the method of Gou\"ezel \cite{GO}. As we have mentioned above, this method was applied in \cite{ANV} in the annealed setting, while  in \cite{Atnip} it was applied for  Gordin-Denker systems.
In a final section we also discuss a few extensions such as different types of  mixing base maps like Young towers or Gibbs-Markov maps, application of the method of cumulants for nonconventional sums of the form $S_n=\sum_{m=1}^n\prod_{j=1}^\ell \varphi_j\circ\tau^{q_j(m)}$, for polynomial $q_j(m)$, as well as extension of the results for different class of random expanding maps (the ones in \cite{MSU}).

	\section{Preliminaries and main results}
\subsection{The random maps}
	We begin by recalling the setup from~\cite{Buz}. Let $(X, \mathcal G)$ be a measurable space endowed with a probability measure $m$ and a notion of a variation $\text{v} \colon L^1(X, m) \to [0, \infty]$ which satisfies
	the following conditions:
	\begin{enumerate}
		\item[(V1)] $\text{v} (th)= |t| \text{v} (h)$;
		\item[(V2)] $\text{v} (g+h)\leq \text{v} (g)+\text{v} (h)$;
		\item[(V3)] $\|h\|_{L^\infty} \le C_{\text{v}}(\|h\|_1+\text{v} (h))$ for some constant $1\leq C_{\text{v}}<\infty$;
		\item[(V4)] for any $C>0$, the set  $\{h\colon X \to \mathbb R: \|h\|_1+\text{v} (h) \leq C\}$ is $L^1(m)$-compact;
		\item[(V5)] $\text{v}(\mathds 1)=0$, where $\mathds 1$ denotes the function equal to $1$ on $X$;
		\item[(V6)] $\{h \colon X \to \mathbb R_+: \lVert h\rVert_1=1 \ \text{and} \ \text{v} (h)<\infty\}$ is $L^1(m)$-dense in
		$\{h\colon X \to \mathbb R_+: \|h\|_1=1\}$;
		\item[(V7)] for any $f\in L^1(X, m)$ such that $\text{essinf} f>0$, we have \[\text{v}(1/f) \leq \frac{\text{v} (f)}{(\text{essinf} f)^2}.\]
		\item[(V8)] $\text{v} (fg)\leq \|f\|_{L^\infty}\cdot \text{v}(g)+\|g\|_{L^\infty}\cdot \text{v}(f)$;
		\item[(V9)] for $M>0$, $f\colon X \to [-M, M]$ measurable and  every $C^1$ function $h\colon [-M, M] \to \mathbb C$, we have
		$\text{v}(h\circ f)\leq\|h'\|_{L^\infty} \cdot \text{v}(f)$.
	\end{enumerate} We define
	\[
	BV=BV(X,m)=\{g\in L^1(X, m): \text{v} (g)<\infty \}.
	\]
	Then, $BV$ is a Banach space with respect to the norm
	\[
	\|g\|_{BV} = \|g\|+ \text{v} (g).
	\]
	\begin{remark}
		Observe that (V3) and (V8) imply that 
		\begin{equation}\label{mc}
			\|fg\|_{BV} \leq C_{\text{v}} \|f\|_{BV} \cdot \|g\|_{BV} \quad \text{for $f, g\in BV$.}
		\end{equation}
		
	\end{remark}
	
	\begin{remark}
		We observe that in \cite{Buz}, assumption (V5) is replaced by the weaker $\text{v}(\mathds 1)<+\infty$. However, for the examples we have in mind, our stronger version is satisfied. In particular, (V5) implies that $\| \mathds 1\|_{BV}=1$.
	\end{remark}
	
	The rest of our setup is almost identical to \cite{DHS}, with a single additional requirement which will be indicated in what follows. 
	Let $(\Omega, \mathcal{F}, \mathbb P, \sigma)$ be a probability space and $\sigma \colon \Omega \to \Omega$  an invertible ergodic measure-preserving transformation. Let  $T_{\omega} \colon X \to X$, $\omega \in \Omega$ be a collection of non-singular  transformations (i.e.\ $m\circ T_\omega^{-1}\ll m$ for each $\omega$) acting   on $X$.  Each transformation $T_{\omega}$ induces the corresponding transfer operator $\mathcal L_{\omega}$ acting on $L^1(X, m)$ and  defined  by the following duality relation
	\begin{equation}\label{Dual}
		\int_X(\mathcal L_{\omega} \phi)\varphi \, dm=\int_X\phi(\varphi \circ T_{\omega})\, dm, \quad \phi \in L^1(X, m), \ \varphi \in L^\infty(X, m).
	\end{equation}
	Thus, we obtain a cocycle of transfer operators  $(\Omega, \mathcal F, \mathbb P, \sigma, L^1(X, m), \mathcal L)$ that we denote by $\mathcal L=(\mathcal L_\omega)_{\omega \in \Omega}$. For $\omega \in \Omega$ and $n\in \mathbb N$, set
	\[
	\mathcal L_\omega^n:=\mathcal L_{\sigma^{n-1} \omega} \circ \ldots \circ \mathcal L_{\sigma \omega} \circ \mathcal L_\omega.
	\]
	
		We recall the  notion of a tempered random variable.
	\begin{definition}
		We say that a measurable map $K\colon \Omega \to (0, +\infty)$ is \emph{tempered} if
		\[
		\lim_{n\to \pm \infty} \frac 1n \log K(\sigma^n \omega)=0, \quad \text{for $\mathbb P$-a.e. $\omega \in \Omega$.}
		\]
	\end{definition}

In this paper we will consider the following assumptions on the random transfer operators.
	\begin{definition}\label{good}
		A cocycle $\mathcal L=(\mathcal L_\omega)_{\omega \in \Omega}$ of transfer operators is said to be \emph{good} if the following conditions hold:
		\begin{itemize}
			\item $\Omega$ is a Borel subset of a separable, complete metric space and $\sigma$ is a homeomorphism. Moreover, $\mathcal L$ is $\mathbb P$-continuous, i.e. $\Omega$ can be written as a countable union of measurable sets such that $\omega \mapsto 
			\mathcal L_\omega$ is continuous on each of those sets;
\item There is a tempered random variable $N(\om)$ so that 
\begin{equation}\label{V}
		 v(g\circ T_\omega) \le N(\omega) v(g), \quad \text{for $\mathbb P$-a.e. $\omega \in \Omega$ and $g\in BV$.}
	\end{equation}			
			
			\item there exists a random variable $C\colon \Omega \to (0, +\infty)$ such that $\log C\in L^1(\Omega, \mathbb P)$ and
			\[
			\|\mathcal L_\omega h\|_{BV}\le C(\omega) \|h\|_{BV}, \quad \text{for $\mathbb P$-a.e. $\omega \in \Omega$ and $h\in BV$;}
			\]
			\item there exist $N\in \bbN$ and random variables $\alpha^N, K^N \colon \Omega \to (0, +\infty)$ such that 
			\[
			\int_\Omega \log  \alpha^N \, d\mathbb P <0, \quad \log K^N \in L^1(\Omega, \mathbb P)
			\]
			and, for $\mathbb P$-a.e. $\omega \in \Omega$ and $h\in BV$,
			\[
			\text{v}(\mathcal L_\omega^N h) \leq \alpha^N (\omega) \text{v}(h)+ K^N(\omega) \|h\|_1;
			\]
			\item for each $a>0$ and $\mathbb P$-a.e. $\omega \in \Omega$, there exist random numbers $n_c(\omega)<+\infty$ and $\alpha_0(\omega), \alpha_1(\omega), \ldots$ such that for every $h\in \mathcal C_a$, 
			\begin{equation}\label{y}
				\text{essinf}_x (\mathcal L_\omega^n h)(x) \ge \alpha_n\|h \|_1 \quad \text{for $n\ge n_c$,}
			\end{equation}
			where 
			\begin{equation}\label{cones}
				\mathcal C_a:=\{ h\in L^\infty (X,m): h\geq 0\, \text{ and }\,\text{v}(h) \leq a\|h\|_1\} ;
			\end{equation}
			\item $\log\left(\text{essinf}_{x\in X}(\cL_\omega \mathds 1)(x)\right)\in L^1(\Omega,\mathbb P)$. 
		\end{itemize}
		
Finally, we say that the cocycle $\mathcal L$ is \textit{uniformly random} if the random variables $C,\al^N, K^N$ and $n_c$ are constants and $\al_n(\om)$ does not depend on $n$ and $\om$. 
	\end{definition}
	
\begin{remark}\label{Rm}
\,
\begin{itemize}
	\item Definition~\ref{good} almost coincides with~\cite[Definition 3]{DHS}, the only difference being the addition of  \eqref{V} (which was considered in \cite[Section 3]{DHS}.)
\item The log-integrability assumption specified at the end of Definition \ref{good} may easily be checked on explicit examples (see e.g. the discussion in \cite[Remark 2.12]{Atnip1}).
\item Furthermore, this assumption implies a certain version of the ``random covering" similar to \eqref{y}, see \cite[Remark 4]{DHS}. 
\end{itemize}
	\end{remark}

	Let us now give examples of systems satisfying our requirements. Our first example  is essentially taken from~\cite{Buz}.
	\begin{example}[Lasota-Yorke cocycles]\label{ex:good}
		Consider $X=[0,1]$, endowed with Lebesgue measure $m$ and the classical notion of variation $\text{v}$. We say that $T:X\to X$ is a piecewise monotonic non-singular map (p.m.n.s map for short) if the following conditions hold:
		\begin{itemize}
			\item T is piecewise monotonic, i.e. there exists a subdivision $0=a_0<a_1<\dots<a_N=1$ such that for each $i\in\{0,\dots,N-1\}$, the restriction $T_i=T_{|(a_i,a_{i+1})}$ is monotonic (in particular it is a homeomorphism on its image).
			\item T is non-singular, i.e. there exists $|T'|:[0,1]\to\mathbb R_+$ such that for any measurable $E\subset (a_i,a_{i+1})$, $m(T(E))=\int_E|T'|dm$.
		\end{itemize}
		The intervals $(a_i,a_{i+1})_{i\in\{0,\dots,N-1\}}$ are called the intervals of $T$. We also set $N(T):=N$ and $\lambda(T):=\text{essinf }_{[0,1]}|T'|$.
		
		We consider a family $(T_\omega)_{\omega\in\Omega}$ of random p.m.n.s as above, and such that $T:\Omega\times [0,1]\to[0,1],~(\omega,x)\mapsto T_\omega(x)$ is measurable.
		Denoting $N_\omega=N(T_\omega)$ and $\lambda_\omega=\lambda(T_\omega)$, we assume that
		\begin{itemize}
			\item The map $\omega\mapsto\left(\text{v}\left(\frac{1}{|T'_\omega|}\right),N_\omega,\lambda_\omega,a_1,\dots,a_{N_\omega-1}\right)$ is measurable.
			\item We have the following expanding-on-average property:
			\begin{equation*}
				\lim_{K\to \infty} \int_\Omega\log \min \left(\lambda_\omega, K\right)~d\mathbb P(\omega) \in (0, +\infty]
			\end{equation*}
			\item The maps $\log (N_\om)$ and $\log^+\left(\frac{N_\omega}{\lambda_\omega}\right)$ are integrable.
			\item The map $\log^+\left(\text{v}\left(\frac{1}{|T_\omega'|}\right)\right)$ is integrable.
			\item $T_\omega$ is covering, i.e. for any interval $I\subset[0,1]$, there exists a random number $n_c(\omega)>0$ such that for any $n\ge n_c$, one has
			\begin{equation}\label{COVER}
				\text{essinf }_{[0,1]} \mathcal L^{n}_\omega(\mathds1_I)>0.
			\end{equation}
			\item $\log\left(\text{essinf }_{x\in X}(\cL_\omega\mathds1)(x)\right)\in L^1(\Omega,\mathbb P)$.
		\end{itemize}
		We will call a cocycle satisfying the previous assumptions an \emph{expanding on average Lasota-Yorke cocycle}. For a countably-valued measurable family $(T_\omega)_{\omega\in\Omega}$ of expanding on average Lasota-Yorke cocycle, the associated cocycle 
		of transfer operators $(\mathcal L_\omega)_{\omega \in \Omega}$ is good (see~\cite{DS}).
	\end{example}
	The following example can be fruitfully compared to a similar one by Kifer~\cite{Kifer-Thermo}.
	\begin{example}
		We consider $X=\mathbb S^1$, endowed with the Lebesgue measure  $m$ and the notion of variation given by $\text{v}(\phi):=\int_X |\phi'|~dm=\|\phi'\|_{L^1}$. We consider a measurable map  $T:\Omega\times X\to X$ such that  $T_\omega:=T(\omega,\cdot)$ is $C^r$, $r\ge 2$. In addition, we make the following assumptions:
		\begin{itemize}
\item There exists a tempered  random variable $N(\om)$ so that \eqref{V} holds true;		
		
			\item The map $\omega\in\Omega\mapsto \left(\int_X\frac{|T_\omega''|}{(T_\omega')^2} dm,\lambda_\omega\right)$ is measurable, where $\lambda_\omega=\inf_{[0, 1]}|T_\omega'|$.
			\item The following expanding on average property holds:
			\begin{equation}\label{hyp:exponaverage}
				\int_\Omega \log(\lambda_\omega)~d\mathds P(\omega)>0.
			\end{equation}
			\item The map $\log\left(\int_X\frac{|T_\omega''|}{(T_\omega')^2} dm\right)$ is $\mathbb P$-integrable.
			\item $\log\left(\text{essinf}_{x\in X}(\cL_\omega\mathds1)(x)\right)\in L^1(\Omega,\mathbb P)$.
		\end{itemize} 
		We call a family $(T_\omega)_{\omega\in\Omega}$ satisfying the previous assumptions a \emph{smooth expanding on average cocycle}. For a family $(T_\omega)_{\omega\in\Omega}$, countably-valued and measurable, of smooth expanding on average cocycle which satisfy \eqref{V}, the associated cocycle 
		of transfer operators $(\mathcal L_\omega)_{\omega \in \Omega}$ is good (see~\cite[Example 16]{DS}). We note that our expansion on average condition \eqref{hyp:exponaverage} implies that $\mathbb P$-a.s, $T_\omega$ has non-vanishing derivative, hence is a local diffeomorphism and a monotonic map of the circle. As  noted in \cite[Example 6]{DHS}, smooth expanding on average cocycles satisfy a stronger version of the random covering property (which by \cite[Remark 0.1]{Buz} implies the one formulated in~\eqref{COVER}): for each non-trivial interval $I\subset X$, for $\mathbb P$-a.e $\omega\in\Omega$, there is a $n_c:=n_c(\omega,I)<\infty$ such that for all $n\ge n_c$,
		\[
		T_\omega^n(I)=X.
		\]
	\end{example}

\subsection{The one dimensionality of the top Oseledets space: a summary of known results}
In this section we recall two results from \cite{DHS} that will be in constant use in the course of the proofs of all of our results. 
\begin{theorem}[\cite{DHS} Theorem 12]\label{T}
	Let $\mathcal L=(\mathcal L_\omega)_{\omega \in \Omega}$ be a good cocycle of transfer operators. Then, the following holds:
	\begin{itemize}
			\item there exists an essentially unique measurable family $(h_\omega)_{\omega \in \Omega}\subset BV$ such that $h_\omega \ge 0$, $\int_X h_\omega \, dm=1$ and
			\[
			\mathcal L_\omega h_\omega=h_{\sigma \omega}, \quad \text{for $\mathbb P$-a.e. $\omega \in \Omega$;}
			\]
			\item there is a random variable $\ell:\Omega\to(0,+\infty)$ such that for $\mathbb P$-a.e. $\omega \in \Omega$, 
			\begin{equation}\label{ell}
				h_\omega \ge \ell(\omega) \quad \text{$m$-a.e.;}
			\end{equation}
			\item for $\mathbb P$-a.e. $\omega \in \Omega$, 
			\begin{equation}\label{split}
				BV=\text{span}\{h_\omega\} \oplus BV^0,
			\end{equation}
			where \[ BV^0=\bigg \{ h\in BV: \int_X h\, dm=0 \bigg \}; \]
			\item $\omega \mapsto \|h_\omega\|_{BV}$ is tempered;
			\item there exist $\lambda >0$ and  for each $\epsilon >0$, a  tempered random variable $D=D_\epsilon \colon \Omega \to (0, +\infty)$ such that for $\mathbb P$-a.e. $\omega \in \Omega$ and $n\in \mathbb N$,
			\begin{equation}\label{Exp1}
				\| \mathcal L_\omega^n \Pi(\omega) \|_{BV} \le D(\omega)e^{-\lambda n}
			\end{equation}
			and 
			\begin{equation}\label{est2x}
				\| \mathcal L_\omega^n(\text{Id}- \Pi(\omega)) \|_{BV} \le D(\omega)e^{\epsilon n},
			\end{equation}
			where $\Pi(\omega) \colon BV \to BV^0$ is a projection associated to the splitting~\eqref{split}.
			
			Finally, for uniformly random cocycles the random variables $\ell(\om)$ and $D(\om)$ can be replaced with positive constants and $\om\to\|h_\om\|_{BV}$ is a bounded random variable. 
		\end{itemize}
\end{theorem}

\begin{corollary}[\cite{DHS} Corollary 13]\label{Cor}
		Let $\mathcal L=(\mathcal L_\omega)_{\omega \in \Omega}$ be a good cocycle of transfer operators. 
		
		Then, the following holds:
		\begin{itemize}
			\item If $(h_\omega)_{\omega \in \Omega}\subset BV$  is given by Theorem~\ref{T}, then 
			\begin{equation}\label{a}
				\omega \mapsto \| 1/h_\omega\|_{BV} \  \text{is tempered.}
			\end{equation}
			\item For $\mathbb P$-a.e. $\omega \in \Omega$, 
			\begin{equation}\label{split2}
				BV=span\{\mathds 1\} \oplus BV_\omega^0,
			\end{equation}
			where 
			\[
			BV_\omega^0=\bigg \{ h\in BV: \int_X h\, d\mu_\omega=0 \bigg \},
			\]
			and $d\mu_\omega=h_\omega dm$, $\omega \in \Omega$;
			\item  there exist $\lambda' >0$ and  a  tempered random variable $\tilde D\colon \Omega \to (0, +\infty)$ such that for $\mathbb P$-a.e. $\omega \in \Omega$ and $n\in \mathbb N$,
			\begin{equation}\label{Exp2}
				\| L_\omega^n \tilde \Pi(\omega) \|_{BV} \le \tilde  D(\omega)e^{-\lambda' n}
			\end{equation}
			\begin{equation}\label{est22}
				\| L_\omega^n(\text{Id}- \tilde \Pi(\omega)) \|_{BV} \le \tilde  D(\omega),
			\end{equation}
			where $\tilde \Pi(\omega) \colon BV \to BV_\omega^0$ is a projection associated to the splitting~\eqref{split2}, and 
			\[
			L_\omega^n g =\mathcal L_\omega^n(g h_\omega) /h_{\sigma^n \omega}, \quad g\in BV, \ n\in \mathbb N.
			\]
			
			Finally, for uniformly random cocycles the random variable $\tilde D(\om)$ can be replaced with a positive constant.
		\end{itemize}
	\end{corollary}
	
	Since $\cL_\om h_\om=h_{\sigma\om}$ and $\cL_\om$ satisfy the duality relation \eqref{Dual} the measure $\mu_\om$ satisfies that for $\bbP$-a.e. $\om$ we have $(T_\om)_*\mu_\om=\mu_{\sig_\om}$. Thus $\mu_\om$ gives raise to a $T$-invariant probability measure $\mu$ on $\Om\times X$ so that 
$$
\mu(A\times B)=\int_{A}\mu_\om(B)d\bbP(\om)=\int_{A\times B}h(\om,x)d\bbP(\om)dm(x)
$$
for every measurable sets $A$ in $\Om$ and $B$ in $X$, where $h(\om,x)=h_\om(x)$.

\subsection{Main results: limit theorems for mixing base maps}
\subsection{The observable}
Let us take a measurable $\varphi:\Om\times X\to\bbR$ so that $\int \varphi d\mu=0$. 
Let $\tilde K(\om)$ be the tempered random variable defined by 
$$
\tilde K(\om)=\max\left(D(\om),\tilde D(\om), N(\om),\|1/h_\om\|_{BV}\right)
$$
where $D(\om),\tilde D(\om)$ and $N(\om)$ are specified in the definition of a good cocycles and in Theorem \ref{T} and Corollary \ref{Cor}.
In order to describe our assumptions on the observable $\varphi$, we will need the following classical result (see~\cite[Proposition 4.3.3.]{Arnold}).
	\begin{proposition}\label{PA}
		Let $\tilde K \colon \Omega \to (0, +\infty)$ be a tempered random variable. For each $\epsilon >0$, there exists a tempered random variable $\tilde K_\epsilon \colon \Omega \to (1, +\infty)$ such that
		\[
		\frac{1}{\tilde K_\epsilon (\omega)} \le \tilde K(\omega) \le \tilde K_\epsilon (\omega) \quad \text{and} \quad \tilde K_\epsilon(\omega)e^{-\epsilon |n|} \le \tilde K_\epsilon (\sigma^n \omega) \le \tilde K_\epsilon (\omega) e^{\epsilon |n|},
		\]
		for $\mathbb P$-a.e. $\omega \in \Omega$ and $n \in \mathbb Z$.
	\end{proposition}
Next, using the notations of Proposition \ref{PA} let $K(\om)=\tilde K_\ve(\om)$ for some $\ve<\la''/3$, where $\la''=\min(\la,\la')$, and $\la$ and $\la'$ care specified in Theorem \ref{T} and Corollary \ref{Cor}, respectively. 
\begin{remark}\label{R}
From now on we will replace both $\la$ and $\la'$ by their minimum, which for notational convenience will be denoted by $\la$.
\end{remark}

In what follows we will consider an observable $\varphi:\Om\times X\to\bbR$ satisfying the scaling condition
\begin{equation}\label{Scale}
\esssup_{\om\in\Om} (K(\om)\|\varphi_\om\|_{BV})<\infty
\end{equation}
which was first introduced in \cite{DS}.
In the uniformly random  case $\tilde K(\om)$ (and hence $K(\om)$) can be replace by a positive constant, and so the scaling condition reads 
$$
\esssup_{\om\in\Om} \|\varphi_\om\|_{BV}<\infty.
$$

The main goal in this paper is to obtain limit theorems for the sequence of functions
$$
S_n=S_n\varphi=\sum_{j=0}^{n-1}\varphi\circ \tau^j
$$
under certain mixing assumptions on the driving system $(\Om,\cF,\bbP,\sigma)$ and the above assumptions on the observable $\varphi$.
\begin{remark}
For expanding on the average maps the scaling condition \eqref{Scale} is necessary for limit theorems, see \cite[Appendix]{DHS}. In any case, our results are also new in the uniformly random  case, and the readers who would prefer can just consider this case together with the assumption that $\esssup_{\om\in\Om}\|\varphi_\om\|_{BV}<\infty$.
\end{remark}

\subsection{Limit theorems}
Let us first introduce our assumptions on the base map.
Let $(\xi_n)$ be a two sided stationary sequence taking values on some measurable space $\cY$. We assume here that $(\Om,\cF,\bbP,\sigma)$ is the corresponding shift system. Namely, $\Omega=\cY^\bbZ$,  $\sig((\om_j)_j)=(\om_{j+1})_j$ is the left shift and if $\pi_0:\Omega\to\cY$ denotes the $0$-th coordinate projection, then $(\xi_n)$ has the same distribution as $(\pi_0\circ\sigma^n)$. We also assume that $T_\om=T_{\om_0}$ and $\varphi(\om,\cdot)=\varphi(\om_0,\cdot)$ depend only on $0$-th coordinate $\om_0$ of $\om$.

\subsubsection{\textbf{Limit theorems for stretched exponentially fast $\alpha$-mixing driving processes}}
\,
\,

Let $(\Om_0,\mathscr F,\textbf{P})$ be the probability space on which $(\xi_n)$ is defined.
We recall that the $\al$-mixing (dependence) coefficient between two sub-$\sig$-algebras $\cG,\cH$ of $\mathscr F$ is given by 
$$
\al(\cG,\cH)=\sup\{|\textbf{P}(A\cap B)-\textbf{P}(A)\textbf{P}(B)|: A\in\cG, B\in\cH\}.
$$
The $\al$-dependence coefficients of $(\xi_n)$ are defined by 
\begin{equation}\label{al def}
\al_n=\sup_{k}\al(\mathscr F_{-\infty,k},\mathscr F_{k+n,\infty})=\al(\mathscr F_{-\infty,0},\mathscr F_{n,\infty})
\end{equation}
where $\mathscr F_{-\infty,k}$ is the $\sigma$-algebra generated by $\xi_j, j\leq k$ and $\mathscr F_{k+n,\infty}$ is generated by $\xi_j, j\geq k+n$. The last equality holds true due to stationarity.
Let us consider the following class of mixing assumptions on the base map:
\begin{assumption}[Stretched exponential $\alpha$ mixing rates]\label{AsAlMix}
There exist positive constants $c_1,c_2$ and $\eta$ so that $\al_n\leq c_1e^{-c_2n^\eta}$ for every $n$.
\end{assumption}

Our first result concerns the variance of $S_n$ and the central limit theorem (with rates).
\begin{theorem}\label{CLT}
Suppose that the cocycle $\cL$ is good.
Let $\varphi$ be an observable so that  $\|\varphi\|_K:=\esssup_{\om\in\Om}(K(\om)\|\varphi_\om\|_{BV})<\infty$, where $\varphi_\om=\varphi(\om,\cdot)$. Suppose that $\sum_n n\al_n<\infty$. Then the limit 
$$
s=\lim_{n\to\infty}n^{-1/2}\|S_n-\bbE[S_n]\|_{L^2(\mu)}
$$
exists and it vanishes if and only if $\varphi=r\circ\tau-r$ for some $r\in L^2(\mu)$. If in addition Assumption \ref{AsAlMix} is satisfied then $n^{-1/2}S_n$ converges in distribution to $sZ$, where $Z$ is a standard normal random variable. Moreover, there is a constant $C>0$ so that for all $n\in\bbN$,
\begin{equation}\label{BErate}
\sup_{t\in\bbR}\left|\mu(S_n-\bbE[S_n]\leq ts\sqrt n)-\Phi(t)\right|\leq Cn^{-\frac1{2+4\gamma}}
\end{equation}
where $\gamma=1/\eta$  and $\Phi$ is the standard normal distribution function. The constant $C$ depends only on $c_1,c_2,\eta$, $\|\varphi\|_K$ and the constant $C_{\text{v}}$ (from the definition of the variation $\text{v}(\cdot)$), and an explicit formula for $C$ can be recovered from the proof.
\end{theorem}
The proof of  Theorem \ref{CLT} appears in Section \ref{CLT pf}.
As discussed in Sections \ref{skew} and \ref{New}, when the quenched CLT holds true with a deterministic centering, then  the CLT for the skew product follows by integration. This was the approach for the CLT in \cite{ANV}, but in the setup of this paper the function $\varphi$ and the measure $\mu_\om$ depend on $\om$, and so the quenched CLT only holds with fiberwise centering. Thus, the novelty of Theorem \ref{CLT} is that the CLT is obtained for the skew product beyond  the annealed case considered in \cite{ANV}. Moreover, Theorem \ref{BErate} also strengthens the CLT in \cite{ETDS}, since our maps $T_\om$ are not uniformly expanding, and the observable $\varphi$ is not fiberwise centered.

Next, let us discuss our results concerning moderate deviations and exponential concentration inequalities.

\begin{theorem}\label{Thm:ModDevNonc}
Suppose that the cocycle $\cL$ is good, and 
let $\varphi$ be an observable so that $\|\varphi\|_K=\esssup_{\om\in\Om}(K(\om)\|\varphi_\om\|_{BV})<\infty$. Let Assumption \ref{AsAlMix} hold and 
set $\gam=\frac1\eta$. Then there exist constants $a_1,a_2>0$ so that for every $x>0$ and $n\in\bbN$,
 \begin{equation}\label{FirstExpCon}
P(S_n-\bbE[S_n]\geq x)\leq \exp\Big(-\frac{x^2}{2(a_1+a_2xn^{-\frac1{2+4\gam}})^{\frac{1+2\gam}{1+\gam}}}\Big).
\end{equation}
All the constants depend only on $c_1,c_2,\eta$, $\|\varphi\|_K$ and $C_{\text{v}}$ from the definition of the variation $\text{v}(\cdot)$, and an explicit formula for them can be recovered from the proof.
\end{theorem}

We will also prove the following:

\begin{theorem}\label{Thm2.4}
Suppose that the cocycle $\cL$ is good, and
let $\varphi$ be an observable so that $\|\varphi\|_K=\esssup_{\om\in\Om}(K(\om)\|\varphi_\om\|_{BV})<\infty$. Let Assumption \ref{AsAlMix} hold and set $\gam=\frac1\eta$. Let us also assume that the asymptotic variance $s^2$ is positive.
\vskip0.2cm
(i) Set 
$v_n=\sqrt{\mathrm{Var}(S_n)}$ and when $v_n>0$ also set
$Z_n=\frac{S_n-\bbE[S_n]}{v_n}$.
Let $\Phi$ be the standard normal distribution function. 
Then  there exist 
constants $s_3,s_4,s_5>0$ 
so that for every $n\geq a_3$ we have $v_n>0$ and for every $0\leq x<a_4 n^{\frac1{2+4\gam}}$,
\begin{eqnarray}\label{ModDev3}
\left|\ln\frac{P(Z_n\geq x)}{1-\Phi(x)}\right|\leq a_5(1+x^3)
n^{-\frac1{2+4\gam}}\,\,\text{ and}\\
\left|\ln\frac{P(Z_n\leq -x)}{\Phi(-x)}\right|\leq a_5(1+x^3)
n^{-\frac1{2+4\gam}}.\nonumber
\end{eqnarray}
The constants $a_4,a_5$ depend only on $c_1,c_2,\eta$, $\|\varphi\|_K$ and $C_{\text{v}}$, and an explicit formula for them can be recovered from the proof.
\vskip0.2cm

(ii) Let $a_n,\,n\geq 1$ be a sequence of real numbers so that 
\[
\lim_{n\to\infty}a_n=\infty\,\,\text{ and }\,\,\lim_{n\to\infty}{a_n}{n^{-\frac{1}{2+4\gam}}}=0.
\]
Then the sequence  $W_n=(sn^{\frac12}a_n)^{-1}S_n,\,n\geq1$ satisfies the moderate deviations principle
with  speed $s_n=a_n^2$ and the rate function $I(x)=\frac{x^2}2$. Namely, for every Borel measurable set $\Gamma\subset\bbR$,
$$
-\inf_{x\in\Gamma^o}I(x)\leq \liminf_{n\to\infty}\frac 1{a_n^2}\ln\mu(W_n\in\Gamma)\leq \limsup_{n\to\infty}\frac 1{a_n^2}\ln\mu(W_n\in\Gamma)\leq -\inf_{x\in\overline{\Gamma}}I(x)
$$
where $\Gamma^o$ is the interior of $\Gamma$ and $\overline{\Gamma}$ is its closure.
\end{theorem}

We also obtain the following Roesenthal type moments estimates.

\begin{theorem}\label{MomThm}
Suppose that $\cL$ is a good cocycle.
If $\|\varphi\|_K=\esssup_{\om\in\Om}(K(\om)\|\varphi_\om\|_{BV})<\infty$, then under  Assumption \ref{AsAlMix}, there exist a constant $c_0$ so that with $\gamma=1/\eta$ 
 for every integer $p\geq1$ we have 
\begin{equation}\label{Rosen}
\big|\bbE[(S_n-\bbE[S_n])^p]-(\mathrm{Var}(S_n))^{\frac p2}\bbE[Z^p]\big|\leq 
(c_{0})^p(p!)^{1+\gam}\sum_{1\leq u\leq\frac {p-1}2}n^u\frac{p^u}{(u!)^2}=O(n^{[(p-1)/2]})
\end{equation}
where  $Z$ be a standard normal random variable. In particular, $\|S_n-\bbE[S_n]\|_{L^p}=O(\sqrt n)$ for every $p$.
As in the previous theorems, the constant $c_0$ depends (explicitly) only on  on $c_1,c_2,\eta$, $\|\varphi\|_K$ and $C_{\text{v}}$.
\end{theorem}
We remark that Theorem \ref{MomThm} provides another proof of the CLT by the method of moments. Indeed, if $s^2>0$ then it follows that for every integer $p\geq1$ the $p$-th moment of $(S_n-\bbE[S_n])n^{-1/2}s^{-1}$ converges to $\bbE[Z^p]$, where $s^2$ is the asymptotic variance. In fact, for even $p$'s we get the convergence rate $O(n^{-1/2})$, while for odd $p$'s we get the rate $O(n^{-1})$.  

\begin{remark}
The proofs of Theorems \ref{Thm:ModDevNonc},  \ref{Thm2.4} and \ref{MomThm} appear in Section \ref{MDP pf}.

Theorems \ref{Thm:ModDevNonc},  \ref{Thm2.4} and \ref{MomThm} are well established for sufficiently fast mixing (in the probabilistic sense) sequences of random variables, where one of the most notable methods of proof is the, so-called, method of cumulants (see \cite{Stat}). 
For random dynamical systems, a  moderate deviations principle was obtained in \cite{DH1}, using a random complex Perron-Frobenius theorem.  In the setup of \cite{ANV}, annealed (local) large deviations principles and exponential concentration inequalities were obtained for iid maps, and we expect that for independent maps the methods in \cite{ANV} will yield results like Theorems \ref{Thm:ModDevNonc}, \ref{Thm2.4} and \ref{MomThm} as well.
The novelty in Theorems \ref{Thm:ModDevNonc},  \ref{Thm2.4} and \ref{MomThm} is that we show how to apply the method of cumulants in the context of skew products with non independent fiber maps, which results in concentration inequalities, moderate deviations principles and Guassian moment estimates  beyond the annealed setup \cite{ANV}.
\end{remark}

Finally,  let us consider the random function $\cS_n(t)=n^{-1/2}\big(S_{[nt]}-\bbE[S_{nt}]\big)$ on $[0,1]$. We also obtain a  functional CLT.
\begin{theorem}\label{FCLT}
Let $\cL$ be a good cocycle.
Suppose that $\esssup_{\om\in\Om}(K(\om)\|\varphi_\om\|_{BV})<\infty$ and that Assumption \ref{AsAlMix} holds true. Then the random function $\cS_n$ converges in distribution towards the distribution of $\{sW_t\}$ where $W$ is a standard Brownian motion (restricted to $[0,1]$) and $s^2$ is the asymptotic variance.
\end{theorem}
\begin{remark}
The proof of Theorem \ref{FCLT} appears in Section \ref{FCLT sec}.
In \cite{ANV} an almost sure invariance principle (ASIP) was obtained, which yields the functional CLT. In the next section, using different mixing coefficients for the base map we will obtain an ASIP for the more general skew products considered in this paper. However Theorem \ref{FCLT} shows that the functional CLT already holds true for stretched exponential $\al$ mixing base maps.
\end{remark}

\subsubsection{\textbf{An almost sure invariance principle and exponential concentration inequalities for $\phi$ and $\psi$-mixing driving processes (via martingale methods)}}
\,
\,

Let $(\Om_0,\mathscr F,\textbf{P})$ be the probability space on which $(\xi_n)$ is defined.
We recall that the $\phi$-mixing  and $\psi$ (dependence) coefficient between two sub-$\sig$-algebras $\cG,\cH$ of $\mathscr F$ is given by 
$$
\phi(\cG,\cH)=\sup\left\{|\textbf{P}(B|A)-\textbf{P}(B)|: A\in\cG, B\in\cH, \textbf{P}(A)>0\right\}
$$
and 
$$
\psi(\cG,\cH)=\sup\left\{\left|\frac{\textbf{P}(A\cap B)}{\textbf{P}(A)\textbf{P}(B)}-1\right|: A\in\cG, B\in\cH, \textbf{P}(A)\textbf{P}(B)>0\right\}.
$$
The reverse $\phi$-mixing coefficients of $(\xi_n)$ are defined by 
\begin{equation}\label{phi def}
\phi_{n,R}=\sup_{k}\phi(\mathscr F_{k+n,\infty},\mathscr F_{-\infty,k})=\phi(\mathscr F_{n,\infty},\mathscr F_{-\infty,0})
\end{equation}
while the $\psi$-mixing coefficients of $(\xi_n)$ are defined by 
\begin{equation}\label{psi def}
\psi_{n}=\sup_{k}\psi(\mathscr F_{-\infty,k},\mathscr F_{k+n,\infty})=\psi(\mathscr F_{-\infty,0},\mathscr F_{n,\infty})
\end{equation}
where $\mathscr F_{-\infty,k}$ is the $\sigma$-algebra generated by $\xi_j, j\leq k$ and $\mathscr F_{k+n,\infty}$ is generated by $\xi_j, j\geq k+n$.
It is clear from the definitions of the mixing coefficients that 
$$
\al_n\leq \phi_{n,R}\leq \psi_n.
$$

\begin{theorem}[Exponential concentration and maximal inequalities]\label{Mart}
Let $\cL$ be a good cocycle.
Suppose the observable satisfies that $\esssup_{\om\in\Om}(K(\om)^2\|\varphi_\om\|_{BV})<\infty$. 

Let $\cF_0$ be the $\sigma$ algebra generated by  the map $\pi(\om,x)=((\om_j)_{j\geq 0},x)$, namely the one generated by  $\cB$ and the coordinates with non-negative indexes in the $\om$ direction. 
If either $\text{ess-inf}\inf_{x} h_\om(x)>0$ and $\sum_{n}\phi_{n,R}<\infty$ or $\sum_{n}\psi_{n}<\infty$ then there is an $\cF_0$-measurable function $\chi\in L^\infty(\mu)$ so that if we set $u=\varphi+\chi\circ\tau-\chi$ then $(u\circ\tau^n)$ is a reverse martingale difference with respect to the  reverse filtration $\{\tau^{-n}\cF_0\}$. As a consequence:
\vskip0.2cm
(i) There are constants $a_1,a_2,a_3>0$  so that the following exponential concentration inequality holds true: for every $t>0$ we have 
\begin{equation}\label{ExConc}
\bbP(|S_n-\bbE[S_n]|\geq tn+a_1)\leq a_2e^{-a_3 nt^2}.
\end{equation}
The constants $a_1,a_2,a_3$ depend  only on  $\tilde\Phi=\sum_{n}\phi_{n,R}<\infty$ and $c$ (or $\tilde\Psi=\sum_{n}\psi_{n}<\infty$), the constant $C_{\text{v}}$ and $\|\varphi\|_{K,2}=\esssup_{\om\in\Om}(K(\om)^2\|\varphi_\om\|_{BV})$, and an explicit formula for them can be recovered from the proof.
\vskip0.2cm

(ii) For every $p\geq 2$ we have 
\begin{equation}\label{MomMart}
\left\|\max_{1\leq k\leq n}|S_k-\bbE[S_k]|\right\|_{L^p}\leq C_pn^{1/2}
\end{equation}
where $C_p>0$ is a a constant (which can be recovered from the proof and depends only on $p$ and the above constants).
\end{theorem}
The proof of Theorem \ref{Mart} appears in Section \ref{MarSec}. Let us note that
once the martingale coboundary representation $\varphi=u+\chi-\chi\circ\tau$ is established, Theorem \ref{Mart} (i) follows from the Azuma–Hoeffding inequality together with Chernoff's bounding method and Theorem \ref{Mart} (ii) follows from, the so called, Rio's inequality \cite{Rio} (see \cite[Proposition 7]{MPU}). 

To obtain the martingale coboundary representation we show that 
if $\cK$ is the  transfer operator\footnote{Namely the one satisfying the duality relation
$$
\int (\cK g)\cdot fd\mu=\int g\cdot (f\circ\tau),\,\, g\in L^1(\Om\times X,\cF_0,\mu), f\in L^\infty(\Om\times X,\cF_0,\mu).
$$} corresponding to the system $(\Om\times X,\cF_0,\mu,\tau)$ then 
\begin{equation}\label{I}
\|\cK^n\varphi-\mu(\varphi)\|_{L^\infty}=O_{n\to\infty}(\del^n+\gamma_{[n/2]})
\end{equation}
where $\gamma_n$ is either $\psi_{n}$ or $\phi_{n,R}$, depending on the case, and $\del\in(0,1)$. 
Once this is established we can take 
$$
\chi=\sum_{n\geq 1}\cK^n\varphi.
$$ 
The proof of \eqref{I} is given in Proposition \ref{MainL} (i).

Our next result is an almost sure invariance principle.
\begin{theorem}[ASIP]\label{ASIP Mart}
Let $\cL$ be a good cocycle, and
suppose that the observable satisfies that $\esssup_{\om\in\Om}(K(\om)^2\|\varphi_\om\|_{BV})<\infty$.

When  $\text{ess-inf}\,\inf_{x} h_\om(x)>0$ we set $\gamma_n=\phi_{R,n}$, while otherwise we set $\gamma_n=\psi_n$. In both  cases, assume that 
$$
\sum_{n\geq 2}n^{5/2}(\log n)^3\gamma_n^4<\infty\,\,\text{ and }\,\,\sum_{n\geq 2}n(\log n)^3\gamma_n^2<\infty
$$
and 
$$
\sum_{n\geq 2}\frac{(\log n)^3}{n^2}\left(\sum_{k=0}^n(k+1)\gamma_k\right)^2<\infty.
$$
Then the limit 
$$
s^2=\lim_{n\to\infty}\frac{1}n\bbE\left[\big(S_n-\bbE[S_n]\big)^2\right]
$$
exists and the following version of the almost sure invariance principle holds true: there is a coupling of $(\varphi\circ\tau^n)$ with a sequence of iid Gaussian random variables $Z_j$ with zero mean and variance $s^2$ so that 
$$
\sup_{1\leq k\leq n}\left|(S_k-\bbE[S_k])-\sum_{j=1}^k Z_j\right|=O(n^{1/4}(\log n)^{1/2}(\log\log n)^{1/4}),\,\,\text{ almost surely}.
$$
\end{theorem}
\begin{remark}
The ASIP implies the functional CLT, see \cite{PS}. Thus, Theorem \ref{ASIP Mart} yields better results than Theorem \ref{FCLT} for $\phi_R$ or $\psi$ mixing driving sequences (which are not necessarily stretched-exponentially mixing).
\end{remark}

The proof of Theorem \ref{ASIP Mart} appears in Section \ref{MarSec}, and it relies on an application of \cite[Theorem 3.2]{CM}. In addition to \eqref{I}, in order to apply \cite[Theorem 3.2]{CM} we will show that for all $1\leq i,j\leq n$ we have
\begin{equation}\label{E}
\left\|\cK^i(\bar\varphi \cK^j\bar\varphi)-\mu\big(\cK^i(\bar\varphi \cK^j\bar\varphi)\big)\right\|_{L^\infty}=O_{n\to\infty}(\del^n+\gamma_n)
\end{equation}
where $\bar\varphi=\varphi-\mu(\varphi)$ and 
 $\del$ and $\gamma_n$  are as in \eqref{I}. The proof of \eqref{E} is given in Proposition \ref{MainL} (ii).

\begin{remark}
As discussed in Section \ref{SM}, the martingale-coboundary decomposition in Theorem \ref{Mart} (and its consequences) is comparable with the annealed case \cite{ANV}, and the main novelty is that we obtain it for more general skew products and functions $\varphi$ which depend on $\om$. Moreover, we do not assume that all $T_\om$ preserve the same absolutely continuous probability measure. The ASIP's we obtain are comparable to ASIP's in \cite{ANV} and \cite{Atnip} (see the discussion in Section \ref{SM}).
\end{remark}

\subsubsection{\textbf{A vector valued almost sure invariance principle in the uniformly random case for exponentially fast $\alpha$-mixing base maps}}
\,
\,

Let us take a vector-valued measurable function $\varphi=(\varphi_1,...,\varphi_d):\Om\times X\to\bbR^d$ so that $\varphi_\om=\varphi(\om,\cdot)$ depend on $\om$ only through $\om_0$ and $\esssup_{\om\in\Om}(K(\om)\|\varphi_{\om,i}\|_{BV})<\infty$ for all $1\leq i\leq d$. Let us also assume that $\mu(\varphi_i)=0$ for every $i$. Set $S_n=\sum_{j=0}^{n-1}\varphi\circ\tau^j$.
\begin{theorem}\label{ASIP GOU}
Suppose that $\al_n=O(\al^n)$ for some $\al\in(0,1)$. Then there is a positive semi definite matrix $\Sig^2$ so that
$$
\Sig^2=\lim_{n\to\infty}\frac1n\text{Cov}(S_n).
$$
Moreover, $\Sig^2$ is positive definite if and only if $\varphi\cdot v \neq r-r\circ\tau$ for all unit vectors v and all $r\in L^2$.

Assume now that there are constants $C>0$ and $\del\in(0,1)$ so that
\begin{equation}\label{App0}
\|\mathcal L_\omega^n \textbf{1}-h_{\sigma^n\omega}\|_{BV}\leq C\delta^n
\end{equation}
namely, that  $K(\om)$ is a bounded random variable.
Then there is a coupling of $(\varphi\circ\tau^n)$ with a sequence of independent Gaussian centered random vectors $(Z_n)$ so that $\text{Cov}(Z_n)=\Sig^2$ and for every $\ve>0$,
$$
\left|(S_n-\bbE[S_n])-\sum_{j=1}^n Z_j\right|=o(n^{1/4+\ve}),\,\text{almost surely}.
$$ 
\end{theorem}

\section{Limit theorems via the method of cumulants for $\al$-mixing driving processes}
We recall next that the $k$-th cumulant of a random variable $W$ with finite moments of all 
orders is given by
\[
\Gam_k(W)=\frac1{i^k}\frac{d^k}{dt^k}\big(\ln\bbE [e^{itW}]\big)\big|_{t=0}.
\]
Note that $\Gam_1(W)=\bbE[W]$, $\Gam_2(W)=\mathrm{Var}(W)$ and that 
$\Gam_k(aW)=a^k\Gam_k(W)$ for any $a\in\bbR$ and $k\geq 1$.

From now on we will assume that $\bbE[S_n]=0$ for all $n$,  namely we will replace $\varphi$ by $\varphi-\mu(\varphi)$.
The main result in this section is the following.
\begin{theorem}\label{CumEst}
Let $\cL$ be a good cocycle,  and suppose that Assumption \ref{AsAlMix} holds true and that $\|\varphi\|_{K}=\esssup_{\om\in\Om}(K(\om)\|\varphi_\om\|_{BV})<\infty$. Then, with $\gamma=1/\eta$,
there exists a constant $c_0$ which depends only on $\|\varphi\|_K$ and the constants from Assumption \ref{AsAlMix} so that for any $k\geq3$,
\[
|\Gam_k(S_n)|\leq n(k!)^{1+\gam}(c_0)^{k-2}.
\]
\end{theorem}

We will prove Theorem \ref{CumEst} by applying the following Proposition \ref{GorcCor}, which appears in \cite{AIHP} as Corollary 3.2.

Let us start with a few preparations.
Let $V$ be a finite set and $\rho:V\times V\to[0,\infty)$ be so that $\rho(v,v)=0$
and $\rho(u,v)=\rho(v,u)$ for all $u,v\in V$. For every $A,B\subset V$ set 
\[
\rho(A,B)=\min\{\rho(a,b): a\in A, b\in B\}.
\] 
We assume here that  
 there exist $c_0\geq1$ and $u_0\geq 0$ so that 
\begin{equation}\label{linear rho}
|\{u\in V: \rho(u,v)\leq s\}|\leq c_0s^{u_0}
\end{equation}
for all $v\in V$ and $s\geq 1$.

Next, let $X_v,\, v\in V$ be a collection of centered random variables 
with finite moments of all orders, and for each $v\in V$ and $t\in(0,\infty]$ let 
$\varrho_{v,t}\in(0,\infty]$ be so that $\|X_v\|_t\leq\varrho_{v,t}$.  
\begin{assumption}\label{X ass}
For some $0<\del\leq\infty$ and all $k\geq 1$, $b>0$ and a finite collection $A_j,\,j\in\cJ$ of (nonempty) subsets of $V$ so that 
$\min_{i\not=j}\rho(A_i,A_j)\geq b$ and $r:=\sum_{j\in\cJ}|A_j|\leq k$ we have
\begin{equation}\label{MixGorc}
\left|\bbE\left[\prod_{j\in\cJ}\prod_{i\in A_j}X_i\right]-\prod_{j\in\cJ} \bbE\left[\prod_{j\in A_j}X_i\right]\right|\leq
(r-1)\Big(\prod_{j\in\cJ}\prod_{i\in A_j}\varrho_{i,(1+\del)k}\Big)\gam_\del(b,k)
\end{equation}
where $\gam_\del(b,r)$ is some nonnegative number which depends only
on $\del,b$ and $r$, and $|\Del|$ stands for the cardinality of a finite set $\Del$.
\end{assumption}

Set $W=\sum_{v\in V} X_v$. 
In the course of the proof of Theorems \ref{CLT}, \ref{Thm:ModDevNonc}, \ref {Thm2.4} and \ref{FCLT} we will need the following general result. 

\begin{proposition}[Corollary 3.2, \cite{AIHP}]\label{GorcCor}
Let \eqref{linear rho} be and Assumption \ref{X ass} be in force.
 Assume also that 
$$
\tilde\gam_\del(m,k):=\max\{\gam_\del(m,r)/r: 1\leq r\leq k\}
\leq de^{-am^\eta}
$$ 
for some
$a,\eta>0$, $d\geq 1$ and all $k,m\geq1$. Then there exists a constant $c$ which depends only
on $c_0,a,u_0$ and $\eta$ so that for every $k\geq 2$, 
\begin{equation}\label{GorEst1}
|\Gam_k(W)|\leq d^k|V|c^k(k!)^{1+\frac{u_0}\eta}\big(M_k^k+M_{(1+\del)k}^k\big)
\end{equation}
where for all $q>0$,
\[
M_{q}=\max\{\varrho_{v,q}:\,v\in V\}\,\,\text{ and }\,\,M_q^k=(M_q)^k.
\]
When the $X_v$'s are bounded and (\ref{MixGorc}) holds true with $\del=\infty$
we can always take $\varrho_{v,t}=\varrho_{v,\infty},\, t>0$ and 
then for any $k\geq2$,
\begin{equation}\label{GorEst1-cor}
|\Gam_k(W)|\leq 2d^k|V|M_\infty^kc^k(k!)^{1+\frac{u_0}\eta}.
\end{equation}
When $\del<\infty$ and there exist $\te\geq0$ and $M>0$ so that
\begin{equation}\label{MomGorch}
(\varrho_{v,k})^k\leq M^k(k!)^{\te}
\end{equation}
for any $v\in V$ and $k\geq1$, then for any $k\geq2$,
\begin{equation}\label{GorEst2}
|\Gam_k(W)|\leq 3C^{\frac{\te}{1+\del}}d^k|V|c^k(1+\del)^k
M^k(k!)^{1+\frac{u_0}{\eta}+\te}
\end{equation}
where $C$ is some absolute constant. 
\end{proposition} 

Theorem \ref{CumEst} will follow from the following result, which is proved in the next section.
\begin{proposition}\label{PropMulti}
For a good cocycle $\cL$ and an observable $\varphi$ satisfying \eqref{Scale} we have the following.
Fix some $n$ and set  $V=\{0,1,...,n-1\}$ and $X_v=\varphi\circ\tau^v$. 
Set also $\rho(x,y)=|x-y|$, and let $t=\delta=\infty$, $\gamma_{\infty}(b,k)=\gamma_b=e^{-(\la-\ve)b/3}+\al_{[b/3]}$.
Then condition \eqref{MixGorc} holds true with the above choices and with 
$$\varrho_{v,\infty}=A_0\max \left(\esssup_{\om\in\Om}(K(\om)\|\varphi_\om\|_{BV}), \|\varphi\|_{L^\infty}\right)$$ 
where $A_0$ is a constant which depends only on $\la-3\ve$ and on the constant $C$ so that $\sup|g|\leq C\|g\|_{BV}$ for every function $g:X\to\bbC$ (and the dependence can be easily recovered from the proof).

If in addition Assumption \ref{AsAlMix} holds then the conditions of Proposition \ref{GorcCor} hold true with $u_0=1$, $c_0=2$ and $\gamma=1/\eta$.
\end{proposition}

 \subsection{Multiple correlation estimates: proof of Proposition \ref{PropMulti}}

Our goal is to show that \eqref{MixGorc} holds true with the desired upper bounds.
We first need the following result.

\begin{lemma}
For every two measurable  functions $g,h$  on $\mathcal Y^{\mathbb N}$ with $g,h\in L^\infty$(w.r.t to the law of $(\xi_n)$) and all $k\in\mathbb Z$ and $n\in\mathbb N$  we have 
\begin{equation}\label{alpha}
\left|\mathbb E[g(...,\xi_{k-1},\xi_k)h(\xi_{k+n},\xi_{k+n+1},...)]-\mathbb E[g(...,\xi_{k-1},\xi_k)]\cdot\mathbb E[h(\xi_{k+n},\xi_{k+n+1},...)]\right|
\end{equation}
$$\leq \frac14 \|g(...,\xi_{k-1},\xi_k)\|_{L^\infty}\|h(\xi_{k+n},\xi_{k+n+1},...)\|_{L^\infty} \alpha_n.$$
\end{lemma}
\begin{proof}
By \cite[Ch.4]{Br}, we have 
$$
\al(\cG,\cH)=\frac14\sup\{\|\bbE[h|\cG]-\bbE[h]\|_{L^1}: h\in L^\infty(\Omega,\cG,\textbf{P}), \|h\|_{L^\infty}\leq 1\}.
$$
Taking $g=g(...,\xi_{k-1},\xi_k)$ and $h=h(\xi_{k+n},\xi_{k+n+1},...)$, $\cG=\mathscr F_{-\infty,k}$ and $\cH=\mathscr F_{k+n,\infty}$ we get 
$$
|\bbE[hg]-\bbE[g]\bbE[h]]=|\bbE[([h|\cG]-\bbE[h])g]|\leq \frac14 \al(\cG,\cH)\|g\|_{L^\infty}\|h\|_{L^\infty}.
$$
\end{proof}

Next, is it clearly enough to prove Proposition \ref{PropMulti} when $\|\varphi\|_{L^\infty}$ and $\esssup_{\om\in\Om}(K(\om)\|\varphi_\om\|_{BV})$ do not exceed $1$, for otherwise we can just divide $\varphi$ by the maximum between the two. 
 Recall also our  assumption  that $K(\omega)e^{-\ve |m|}\leq K(\sigma^{m}\omega)\leq K(\omega)e^{\ve |m|}$ for some $\ve<\lambda/3$ (recall Remark \ref{R}). 
 

The first step in the proof of Proposition \ref{PropMulti} is the following result.
\begin{lemma}[Fiberwise multiple correlation estimates]\label{L1}
Let $B_1,B_2,...,B_m$ be nonempty intervals in the nonnegative integers so that $B_i$ is to the left of $B_{i+1}$ and $B_1$ contains $0$. Let us denote by $d_i$ the gap between $B_i$ to $B_{i+1}$ (namely the distance).
Let us fix some $\omega$ and let $f_i$ be a family of functions so that $K(\sigma^i\omega)\|f_i\|_{BV}\leq 1$ and $\|f_i\|_{L^\infty}\leq1$. Let us define $F_j=F_{B_j,\omega}=\prod_{i\in B_j}f_i\circ T_\omega^i$. Then 
$$
\left|\left(\int \left(\prod_{j=1}^{m}F_j\right)d\mu_\omega\right)-\left(\prod_{j=1}^m\int F_jd\mu_\omega\right)\right|\leq A\sum_{j=1}^{m-1}e^{-(\lambda-\ve) d_j},
$$
where $A=C^2\sup_{d\in\mathbb N}2de^{-(\lambda-\ve)d}$ and $\la$ comes from \eqref{Exp1} and \eqref{Exp2} (recall Remark \ref{R}).
\end{lemma}  
\begin{proof}
The proof will be carried out by induction on $m$.
Let us first prove the lemma in the case $m=2$.
We first note that for all functions $g_0,g_1,...,g_q$ we have 
$$
\text{v}\left(\prod_{k=0}^{q}g_k\circ T_\omega^k\right)\leq\sum_{k=0}^q \left(\prod_{0\leq s<k}\|g_s\|_\infty\right)\cdot\left(\text{v}(g_k\circ T_\omega^k)\right)\cdot\left(\prod_{k<s\leq q}\|g_s\|_\infty\right)
$$
where $\|g_i\|_\infty=\sup\|g_i\|_{L^\infty}$,
and hence
\begin{equation}\label{U}
\left\|\prod_{k=0}^qg_k\circ T_\omega^k\right\|_{BV}\leq\prod_{k=0}^q\|g_k\|_\infty+\sum_{k=0}^q\left(\prod_{0\leq s<k}\|g_s\|_\infty\right)\cdot\left(\prod_{s=0}^{k-1}K(\sigma^s\omega)\text{v}(g_k)\right)\left(\prod_{k<s\leq q}\|g_s\|_\infty\right)
\end{equation}
where we have used \eqref{V},  that $N(\om)\leq K(\om)$ and that 
$$
\left\|\prod_{k=0}^qg_k\circ T_\omega^k\right\|_{L^1}\leq \left\|\prod_{k=0}^qg_k\circ T_\omega^k\right\|_{L^\infty}\leq \prod_{k=0}^q\|g_k\|_\infty.
$$ 
Let us write $B_1=\{0,1,...,d\}$. Taking $g_k=f_k$ for $0\leq k\leq d=q$ and noting that $K(\sigma^s\omega)\|g_{s}\|_\infty\leq C$ for some constant $C$ which depends\footnote{$C$ is a constant which satisfies that $\|g\|_\infty=\sup|g|\leq C\|g\|_{BV}$ for every complex function on $X$.} only the space $X$  
we conclude that 
$$
\|F_1\|_{BV}\leq C(d+1)\leq 2Cd.
$$

Now, if we write $B_2=\{d+n,d+n+1,...,d+n+L\}$ then
$$
\mu_\omega(F_1F_2)=\mu_\omega(F_1\cdot G_2\circ T_\omega^{d+n})=\mu_{\sigma^{n+d}\omega}(G_2 L_\omega^{n+d}F_1 )
$$
where 
$$
G_2=\prod_{u\in B_2}f_u\circ T_{\sigma^u\omega}^{u-n-d}.
$$
By \eqref{Exp2} we have,
$$
\left\|L_\omega^{n+d}F_1-\mu_\omega(F_1)\right\|_{BV}\leq K(\omega)\|F_1\|_{BV}e^{-\lambda(d+n)}\leq 2dCK(\omega)e^{-\lambda(d+n)}.
$$
Therefore, using also that $\mu_\omega$ is an equivariant family and that  (since $n+d\in B_2$)
$$\|G_2\|_{L^\infty}\leq\|f_{n+d}\|_{L^\infty}\leq CK(\sigma^{n+d}\omega)^{-1}$$
we get that
$$
\left|\mu_\omega(F_1F_2)-\mu_\omega(F_1)\mu_{\omega}(F_2)\right|=\left|\mu_{\sigma^{n+d}\omega}(G_2 L_\omega^{n+d}F_1)-\mu_\omega(F_1)\mu_{\sigma^{d+n} \omega}(G_2)\right|$$
$$
=\left|\int (L_\omega^{d+n}F_1-\mu_\omega(F_1))G_2 \, d\mu_{\sigma^{d+n} \omega}\right|\leq 2dCK(\omega)e^{-\lambda(d+n)}\|G_2\|_{L^\infty}
$$
$$\leq 2dCK(\omega)e^{-\lambda(d+n)}K(\sigma^{n+d}\omega)^{-1}\leq 2dC^2 e^{-(\lambda-\ve)(d+n)}=
(2C^2de^{-(\lambda-\ve)d})e^{-(\lambda-\ve)n}.
$$
This proves the lemma for $m=2$.

Next, let us complete the induction step. Let $d$ be the right end point of $B_{m-1}$. Then $d+d_m$ is the left end point of $B_m$ and we can write
$$
\mu_\omega\big(\prod_{k}F_k\big)=\mu_\omega \big(\prod_{k<m}F_k\cdot (G_m\circ T_{\omega}^{d+d_m})\big)=
\mu_{\sigma^{d+d_m}\om}\big(L_\omega^{d+d_m}(\prod_{k<m}F_k)\cdot G_m\big)
$$
where $G_m$ is some function. Now we observe that
$$
\left\|\prod_{k<m}F_k\right\|_{BV}\leq C(d+1)\leq 2Cd
$$
which is proved exactly as in the previous case (even though there are gaps between the blocks $B_j$, we can set $g_i=1$ when $i$ does not belong to one of the $B_j$'s, and then $\text{v}(g_i)=0$). Thus, as in the case $m=2$, we have 
$$
\left|\mu_\omega\left(\prod_{k}F_k\right)-\mu_\omega\big(F_m\big)\mu_\omega\left(\prod_{k<m}F_k\right)\right|\leq (2C^2de^{-(\lambda-\ve)d})e^{-(\lambda-\ve)d_m}.
$$
 The induction is completed by the above inequality, taking into account that $|\mu_\omega(F_m)|\leq 1$.
\end{proof} 

Integrating over $\om$ yields the following corollary of Lemma \ref{L1}.
\begin{corollary}\label{Cor1}
Let $\tau$ be the skew product. Let $B_j, 1\leq j\leq m$ be blocks as in Lemma \ref{L1}.
Set $G_j=\prod_{i\in B_j}\varphi\circ\tau^{i}$. Let us denote by $b_j$ the left end point of $B_j$.
 Then
\begin{equation}\label{EC1}
\left|\int \prod_{j=1}^{m} G_j d\mu-\int\left(\prod_{j=1}^{m}\int\left(\prod_{i\in B_j}\varphi_{\sigma^i\omega}\circ T_{\sig^{b_j}\om}^{i-b_j}
\right)d\mu_{\sigma^{b_j}\omega}\right)d\mathbb P(\omega)\right|\leq A\sum_{j=1}^{d}e^{-\lambda d_j}.
\end{equation} 
\end{corollary}

The next step of the proof is to estimate the second term inside the absolute value on the left hand side of \eqref{EC1}. To obtain appropriate estimates, we first need the following lemma:
\begin{lemma}\label{LemmmaAp}
Let us fix some $k\in\mathbb N$ and set
$$
F_\omega=\prod_{j=0}^k\varphi_{\sigma^k\omega}\circ T_{\omega}^k.
$$
Then for every $n\in\mathbb N$ and for $\mathbb P$ a.e. $\omega$ we have
$$
\left|\mu_\omega(F_\omega)-m(F_\om\mathcal L_{\sigma^{-n}\omega}^n\textbf{1})\right|\leq Ce^{-n(\lambda-\ve)}
$$
where $C$ is such that $\|g\|_{L^\infty}\leq C\|g\|_{BV}$ for every function $g$ on $X$ with bounded variation (recall that such a constant $C$ exists by our assumption on the variation $\text{v}(\cdot)$).
\end{lemma}
\begin{proof}
Using \eqref{Exp1}, that $K(\sig^{-n}\om)\leq e^{\ve n}K(\om)$ and that  $\|F_\omega\|_{L^\infty}\leq \|\varphi_\omega\|_{L^\infty}\leq C\|\varphi_\om\|_{BV}\leq CK(\om)^{-1}$ we obtain that
$$
\left|\mu_\omega(F_\omega)-m(F_\om\mathcal L_{\sigma^{-n}\omega}^n\textbf{1})\right|=
$$
$$
\left|\int(h_\omega-\mathcal L^n_{\sigma^{-n}\omega}\textbf{1})F_
\omega dm\right|\leq CK(\omega)^{-1}\int|h_\omega-\mathcal L^n_{\sigma^{-n}\omega}\textbf{1}|dm
$$
$$
\leq K(\omega)^{-1}e^{-\lambda n}K(\sigma^{-n}\omega)\leq Ce^{-n(\lambda-\ve)}.
$$
\end{proof}

Taking into account that $|\mu_\omega(F_\omega)|\leq1$, that $|m(F_\omega \cL_{\sigma^{-n}\omega}^n\textbf{1})|=|m(F_\om\circ T_{\sig^{-n}\om}^n)|\leq1$ and that $|\prod_{j}\alpha_j-\prod_j\beta_j|\leq\sum_j|\alpha_j-\beta_j|$ for all numbers $\alpha_j,\beta$ so that  $|\al_j|,|\beta_j|\leq1$ we get the following result directly from  Corollary \ref{Cor1} and Lemma \ref{LemmmaAp}.
\begin{corollary}\label{Cor2}
Let $b_j$ be the left end point of the block $B_j$. Let us also set $r_j=d_j/3$ and $r_0=r_1$. Then there exists a constant $A_1>0$ which does not depend on $\om$ or on the blocks so that
in the notations of  Corollary \ref{Cor1} and Lemma \ref{LemmmaAp} we have
$$
\left|\int \prod_{j=1}^{m}G_jd\mu-\int\left(\prod_{j=0}^{d}m(\varphi_{\om,j}\mathcal L^{d_j}_{\sigma^{b_j-d_j}\omega}\textbf{1})\right)d\mathbb P(\omega)\right|\leq A_1\sum_{j=1}^{m-1}e^{-(\lambda-\ve)r_j}
$$ 
where
$$
\varphi_{\omega,j}=\prod_{i\in B_j}\varphi_{\sigma^i\omega}\circ T_{\sigma^{b_j}\omega}^{i-b_j}.
$$
\end{corollary}

Now, we observe that $m(\varphi_{\om,j}\mathcal L^{d_j}_{\sigma^{n_j-d_j}\omega}\textbf{1})$ is a function of $\xi_{b_j-r_j},...,\xi_{b_{j+1}-r_j}$ (i.e. of the coordinates $\om_{b_j-r_j},...,\om_{b_{j+1}-r_j}$). Namely, in distribution it can be written as
$$
m(\varphi_{\om,j}\mathcal L^{d_j}_{\sigma^{n_j-d_j}\omega}\textbf{1})=f_j(\xi_{b_j-r_j},...,\xi_{b_{j+1}-r_j})
$$
for some measurable function $f_j$. Since
$
m(\varphi_{\om,j}\mathcal L^{d_j}_{\sigma^{n_j-d_j}\omega}\textbf{1})=m(\varphi_{\om,j}\circ T_{\sigma^{n_j-d_j}\omega}^{d_j})
$
and $|\varphi_{\om,j}|\leq 1$, we can insure that $|f_j|\leq1$.
 Using \cite[(2.20)]{AIHP} and Corollary \ref{Cor2}  we conclude that:
\begin{corollary}\label{ProMult}
Let $G_j, 1\leq j\leq m$ be as in Corollary \ref{Cor1} (defined by some blocks $B_j$ with gaps $d_j$).
There are constants $A>1$ and $\delta_0\in(0,1)$ which do not depend on the blocks so that 
$$
\left|\int \left(\prod_{j=1}^{m}G_j\right)d\mu-\left(\prod_{j=1}^{m}\int G_jd\mu\right)\right|\leq A\sum_{j=1}^{m}(\delta_0^{r_j}+\alpha([r_j])).
$$ 
\end{corollary}
All that is left is to notice that Corollary \ref{ProMult} is a reformulation of Proposition \ref{PropMulti}, using the notations of this section.

\subsection{Limit theorems via the method of cumulants}

\subsubsection{The CLT: proof of Theorem \ref{CLT}}\label{CLT pf}
First,  by Proposition \ref{PropMulti} we have that \eqref{MixGorc} holds true with the numbers $\varrho_{i,(1+\del)k}$ and $\gamma_\del(b,k)$ specified in Proposition \ref{PropMulti} . By taking $r=2$, $A_1=\{0\}$ and $A_2=\{n\}$ in  \eqref{MixGorc} we see that
$$|\bbE_\mu[\varphi\cdot \varphi^n]|=O( \del^n+\al_{[n/3]})$$ for some $\del\in(0,1)$. Hence, if $\sum n\al_n<\infty$ then $\sum_n n|\bbE_\mu[\varphi\cdot \varphi^n]|<\infty$ and the results concerning the asymptotic variance $s^2$ follow from the general theory of (weakly) stationary processes (see \cite{IBR} and Lemma \ref{L3} below).

Now, suppose that $s^2=\lim_{n\to\infty}\frac1n\text{Var}_\mu(S_n)>0$, where $S_n=S_n\varphi$. To prove the CLT and the convergence rate  \eqref{BErate},by  applying \cite[Corollary 2.1]{Stat}, taking into account Theorem \ref{CumEst}, we get the CLT and the rate \eqref{BErate} for $S_n/\sqrt{\text{Var}(S_n)}$. 	To get the same rate for $S_n/\sqrt n$ we need the following general fact from the theory of stationary real-valued sequences, which for the sake of convenience is stated as a lemma.
\begin{lemma}\label{L3}
Let $Y_n$ be a centered weakly stationary sequence of square integrable random variables. Set $b_n=\bbE[Y_0Y_n]$ and $S_n=\sum_{j=1}^nY_j$. Suppose that $\sum_{k}k|b_k|<\infty$. Then 
$$
\lim_{n\to\infty}\frac 1n\bbE[S_n^2]=b_0+2\sum_{n\geq 1}b_n:=s^2
$$
and 
$$
\left|\frac1n\bbE[S_n^2]-s^2\right|\leq 2n^{-1}\sum_{k=1}^\infty k|b_k|.
 $$
\end{lemma}
Let us give a reminder of the short proof. We have $\frac1n\bbE[S_n^2]=\sum_{k=1}^{n-1}(1-k/n)b_k+b_0$ and so 
$$
\left|\frac1n\bbE[S_n^2]-s^2\right|=\left|2\sum_{k=n}^\infty b_k+2n^{-1}\sum_{k=1}^{n-1}kb_k\right|\leq 
2n^{-1}\left(\sum_{k=n}^\infty k|b_k|+\sum_{k=1}^{n-1}k|b_k|\right)\leq 2n^{-1}\sum_{k\geq 1}k|b_k|.
$$
Using this lemma together with \cite[Lemma 3.3]{BE} with $a=2$ and that 
$$
\left\|\frac{S_n}{\sqrt{\text{Var}(S_n)}}-\frac{S_n}{s\sqrt n}\right\|_{L^2}
=\|S_n\|_{L^2}\left|\frac{1}{\sqrt{\text{Var}(S_n)}}-\frac1{s\sqrt n}\right|=O(n^{1/2})\cdot O(n^{-3/2})=O(n^{-1})
$$
we obtain \eqref{BErate}.

\subsubsection{A moderate deviations principle, stretched exponential concentration inequalities and 	Rosenthal type estimates: proof of Theorems \ref{Thm:ModDevNonc}, \ref{Thm2.4} and \ref{MomThm}.}
\label{MDP pf}
\,
\,

First, Theorem \ref{Thm:ModDevNonc} follows from Theorem \ref{CumEst} and \cite[Lemma 2.3]{Stat}. 
 The estimates \eqref{ModDev3} stated in Theorem \ref{Thm2.4}  follow from Theorem \ref{CumEst}  and \cite[Lemma 2.3]{DE} (which is a consequence of \cite[Lemma 2.3]{Stat}).
  The moderate deviations principle stated in Theorem \ref{Thm2.4} follows from Theorem \ref{CumEst} and \cite[Theorem 1.1]{DE}. We note that the conditions of  \cite[Lemma 2.3]{Stat}, \cite[Lemma 2.3]{DE}  and  \cite[Theorem 1.1]{DE} are certain estimates on the growth rates (in $k$) of the cumulants  $\Gamma_k(S_n)$, and the role of  Theorem \ref{CumEst} is that it shows that the conditions of all of these results are in force in the setup of this paper.

\subsection{A functional CLT via the method of cumulants: proof of Theorem \ref{FCLT}}\label{FCLT sec}
Let us first show that the sequence $\cS_n$ is tight. By Theorem \ref{MomThm} we have that 
$$
\|S_n\|_{4}=O(\sqrt n)
$$
where $\|\cdot\|_4=\|\cdot\|_{L^4}$, and therefore, using also stationarity and the H\"older inequality we get that for all $t_1<t_2\leq r_1<r_2$,
$$
\bbE\left[\left(\cS_{n}(r_2)-\cS_{n}(r_1)\right)^2\left(\cS_{n}(t_2)-\cS_{n}(t_1)\right)^2\right]\leq 
\|\cS_{n}(r_2)-\cS_{n}(r_1)\|_4^2\|\cS_{n}(t_2)-\cS_{n}(t_1)\|_4^2$$$$
\leq C\left(\frac{[r_2n]-[t_1n]}{n}\right)^2.
$$
Thus, by \cite[Ch.15]{Bil},\,  $\cS_n(\cdot)$ is a tight sequence in the Skorokhod space $D[0,1]$.

Now let us show that the finite-dimensional distributions converge. Let us fix some $t_1<t_2<...<t_d$. Set $X_k=\varphi\circ\tau^k$. Next, let us recall the following general fact. Given a vector valued sequence of random variables $Y_n=(Y_{1,n},...,Y_{d,n})$, 
by the multidimensional version of Levi's theorem, in order to show that  $Y_n$ converges in distribution as $n\to\infty$ towards a given random variable $\cZ$,
it is enough to show that 
for every $a\in\bbR^d$ we have 
$$
\lim_{n\to\infty}\bbE[e^{i(a\cdot Y_n)}]=\bbE[e^{i(a\cdot\cZ)}].
$$
Therefore, it is enough to show that any linear combination of $Y_{j,n}, j=1,2,...,d$ converges in distribution towards the corresponding linear combination of  the coordinates of $\cZ$.
Returning to our problem, to obtain the appropriate convergence of the distribution of $(\cS_n(it_j))_{j=1}^{d}$
it is enough to show that any linear combination of $\cS_{n}(t_j)$ converges towards a centered normal random variable with an appropriate variance. More precisely, 
let $a_1,...,a_d\in\bbR$. Then we need to show that $\sum_{j=1}^{d}a_j\cS_{n}(t_j)$ converges in distribution towards a centered normal random variable with variance 
$$
s^2\sum_{j=1}^d\left(a_j+....+a_d\right)^2(t_j-t_{j-1})
$$ 
where $t_0=0$ and $s^2=\lim_{n\to\infty}\frac 1n\bbE[S_n^2]$. We first notice that 
$$
\sum_{j=1}^da_j \cS_{n}(t_j)=n^{-1/2}\sum_{j=1}^d(a_j+...+a_d)\left(S_{[nt_j]}-S_{[nt_{j-1}]}\right)
$$ 
where we set $t_0=0$ and $S_0=0$.
Thus, using stationarity, we have
$$
\bbE\left[\left(\sum_{j=1}^{d}a_j\cS_{n}(t_j)\right)^2\right]=n^{-1}\sum_{j=1}^d\left(a_j+....+a_d\right)^2\bbE[S_{[nt_j]-[nt_{j-1}]}^2]
$$
$$
+2n^{-1}\sum_{1\leq j_1<j_2\leq d}(a_{j_1}+...+a_{j_d})(a_{j_2}+...+a_{j_d})\bbE\left[\big(S_{[nt_{j_2}]}-S_{[nt_{j_2-1}]}\big)\big(S_{[nt_{j_1}]}-S_{[nt_{j_1-1}]}\big)\right].
$$
Now, the first summand on the above right hand side converges to 
$$
s^2\sum_{j=1}^d\left(a_j+....+a_d\right)^2(t_j-t_{j-1}),
$$ 
while the second summand (the double sum) converges to $0$ because $|\bbE[\varphi\cdot \varphi\circ\tau^n]|$ converges to $0$ stretched exponentially fast. Therefore, the asymptotic variance of $\sum_{j=1}^{d}a_j\cS_{n}(t_j)$ has the desired form. Now, let us consider the following array of random variables. Set  
$$Y_k=Y_k^{(n,a_1,...,a_d,t_1,...,t_d)}=(a_1+...+a_j)\varphi\circ\tau^k\,\text{ if }\, [nt_{j-1}]\leq k<[nt_j].$$
Then, 
$$
\sum_{j=1}^d(a_j+...+a_d)\left(S_{[nt_j]}-S_{[nt_{j-1}]}\right)=
\sum_{j=1}^d(a_j+...+a_d)\sum_{s=[nt_{j-1}]}^{[nt_j]-1}\varphi\circ\tau^s
$$
$$
=\sum_{j=1}^d(a_j+...+a_d)\sum_{s=0}^{[nt_d]-1}\bbI([nt_{j-1}]\leq s<[nt_j])\varphi\circ\tau^s=
\sum_{s=0}^{[nt_d]-1}Y_s.
$$
On the other other hand,
arguing as in the proof of Theorem \ref{CumEst} (replacing each appearance of $\varphi\circ\tau^k$ by $Y_k$) we get the same kind of estimates on the cumulants of 
$$
\tilde S_n:=\sum_{s=0}^{[nt_d]-1}Y_s,
$$
that is, there exists a constant $c_0$ which might depend on $t_j$ and $a_j$ so that for every $k$ we have 
$$
|\Gam_k(\tilde S_n)|\leq n(k!)^{1+\gam}(c_0)^{k-2}.
$$
Thus, by applying \cite[Corollary 2.1]{Stat} we get that
$$
\sum_{s=0}^{[nt_d]-1}Y_s^{(n,a_1,...,a_d)}/w_n
$$
converges towards the standard normal distribution,
where $w_n$ is the standard deviation of the numerator. Note that, as we have shown, $w_n^2/n\to s^2\sum_{j=1}^d\left(a_1+....+a_d\right)^2(t_j-t_{j-1})$, which is positive unless either $s=0$ or $ a_1=...=a_d=0$, which are both trivial cases.  Thus, in any case we obtain the desired convergence of the linear combination
$
\sum_{j=1}^d a_j\cS_{n}(t_j)
$
and the proof of Theorem \ref{FCLT} is complete.

\section{Limit theorems via martingale approximation for $\phi$ and $\psi$ mixing driving processes}\label{MarSec}
\subsection{Some expectation estimates using mixing coefficients}
In the course of the proof of Theorem \ref{ASIP Mart} we will need the following two relatively simple 
lemmas.

\begin{lemma}\label{PhiLem}
Let $\cG,\cH$ be two sub-$\sig$-algebras of a given $\sig$-algebra on some space measure space. Let $g$ be a real-valued bounded $\cG$-measurable function and $h$ be an $\cH$-measurable real-valued integrable function. 
Then 
$$
\left|\bbE[hg]-\bbE[h]\bbE[g]\right|\leq\frac12\|h\|_{L^\infty}\|g\|_{L^1}\phi(\cG,\cH)
$$
\end{lemma}
\begin{proof}
By \cite[Ch. 4]{Br} we have
$$
\|\bbE[h|\cG]-\bbE[h]\|_{L^\infty}\leq \frac12\|h\|_{L^\infty}\phi(\cG,\cH)
$$
which clearly implies the lemma.
\end{proof}

The next result is:
\begin{lemma}\label{PsiLem}
Let $\cG,\cH$ be two sub-$\sig$-algebras of a given $\sig$-algebra on some space measure space. Let $g$ a real-valued bounded $\cG$-measurable function and $h$ be an $\cH$-measurable real-valued integrable function. 
Suppose also that $\psi=\psi(\cG,\cH)<1$. Then 
$$
\left|\bbE[hg]-\bbE[h]\bbE[g]\right|\leq 4\|hg\|_{L^1}C_{\psi}\psi
$$
where $C_{\psi}=(1-\psi)^{-1}$.
\end{lemma}
\begin{proof}
By \cite[Ch.4]{Br} we have, 
$$
\|\bbE[h|\cG]-\bbE[h]\|_{L^\infty}\leq \|h\|_{L^1}\psi(\cG,\cH).
$$
Hence 
$$
\left|\bbE[hg]-\bbE[h]\bbE[g]\right|\leq \|h\|_{L^1}\|g\|_{L^1}\psi.
$$
Taking $h,g\geq 0$ we get that 
$$
\left|\bbE[hg]-\bbE[h]\bbE[g]\right|\leq \bbE[h]\bbE[g]\psi.
$$
Thus, 
$$
\bbE[h]\bbE[g]\leq (1-\psi)^{-1}\bbE[hg]=C_{\psi}\bbE[hg].
$$
Therefore, for nonnegative functions we have
$$
\left|\bbE[hg]-\bbE[h]\bbE[g]\right|\leq C_{\psi}\psi\bbE[hg].
$$
Now the general result follows by writing $h=h^+-h^{-}$ and $g=g^{+}-g^{-}$ where $h^{\pm}$ and $g^{\pm}$ are nonnegative functions so that $h^{+}+h^{-}=|h|$ and $g^{+}+g^{-}=|g|$, and using that both $(g,h)\to \bbE[g]\bbE[h]$ and $(g,h)\to\bbE[hg]$ are bilinear in $(g,h)$. 
\end{proof}

\subsection{Convergence of the iterates of the transfer operator with respect to a sub-$\sig$-algebra}
Let $\cF_0$ be the $\sigma$-algebra generated by the map $\pi(\om,x)=((\om_j)_{j\geq 0},x)$, namely the one generated by  $\cB$ and the coordinates with non-negative indexes in the $\om$ direction. Then $(\tau^{-k}\cF_0)_{k\geq0}$ is a decreasing sequence of $\sig$-algebras and $\tau^{-k}\cF_0$ is generated by $\tau^k$ and the coordinates $\om_{j}$ for $j\geq k$. In particular $\tau$ preserves $\cF_0$.

Next, let us define a transfer operator with respect to $\cF_0$. For each function $g\in L^1(\mu)$ there is a unique $\cF_0$-measurable function $G$ so that 
$$
\bbE[g|\tau^{-1}\cF_0]=G\circ\tau.
$$
Let us define $\cK g=G$, where we formally set $G$ to be $0$ outside the image of $\tau$ (if $\tau$ is not onto). Then 
$$
\bbE[g|\tau^{-1}\cF_0]=\cK g\circ\tau.
$$
Notice that for $g\in L^1(\Om\times X,\cF_0,\mu), f\in L^\infty(\Om\times X,\cF_0,\mu)$ we have
$$
\int (\cK g) f\,d\mu=\int (\cK g\circ\tau)f\circ\tau \,d\mu=\int \bbE[g|\tau^{-1}\cF_0]\cdot f\circ\tau \,d\mu=
\int g\cdot f\circ\tau \,d\mu
$$
and therefore $\cK$ can also be defined using the usual duality relation.
That is, it is the transfer operator of $\tau$ with respect to $(\Om\times X,\cF_0,\mu)$.

The proof of Theorems \ref{Mart} and \ref{ASIP Mart} is based on the following result.

\begin{proposition}\label{MainL}
Under the assumptions of Theorems \ref{Mart} and \ref{ASIP Mart}, and when $\mu(\varphi)=0$ we have the following:

(i) We have
\begin{equation}\label{2nd}
\left\|\cK^n\varphi\right\|_{L^\infty}\leq C\big(e^{-(\la-2\ve)n/2}+\psi_{[n/2]}\big):=C\gamma_{2,n}.
\end{equation}
Moreover, if  $h_\om \geq c^{-1}>0$ for some constant $c>1$ then
\begin{equation}\label{1st}
\left\|\cK^n\varphi\right\|_{L^\infty}\leq Cc\big(e^{-(\la-2\ve)n/2}+\phi_{[n/2],R}\big):=C\gamma_{1,n}.
\end{equation}
Here $C=C_{\varphi}$ is a constant having the form $C_{\varphi}=AC_{\text{v}}\esssup_{\om\in\Om}(K(\om)^2\|\varphi_\om\|_{BV})$
where $A$ is an absolute constant and $C_0$ is any constant satisfying $\|g\|_{L^\infty}\leq C_0\|g\|_{BV}$ and 
$\|fg\|_{BV}\leq C_0\|g\|_{BV}\|f\|_{BV}$ for all functions $g,f:X\to\bbC$.
\vskip0.1cm
(ii)  We have
$$
\left\|\cK^i(\varphi \cK^j\varphi)-\mu\big(\cK^i(\varphi \cK^j\varphi)\big)\right\|_{L^\infty}\leq  C\gamma_{2,\max(i,j)}.
$$
If $h_\om \geq c^{-1}>0$ for some constant $c>1$ then
$$
\left\|\cK^i(\varphi \cK^j\varphi)-\mu\big(\cK^i(\varphi \cK^j\varphi)\big)\right\|_{L^\infty}\leq Cc\gamma_{1,\max(i,j)}.
$$
\end{proposition}

\begin{proof}[Proof of Theorems \ref{Mart} and \ref{ASIP Mart} based on Proposition \ref{MainL}]
First, theorem \ref{Mart} (i) follows since if we set $\chi=\sum_{n=1}^\infty K^n\varphi$ and $u=\varphi+\chi\circ\tau-\chi$, then $\|\chi\|_{L^\infty}<\infty$ and
$(u\circ \tau^n)$ is a reverse martingale difference with respect to the reverse filtration $\{\tau^{-n}\cF_0\}$. Moreover, the differences $u\circ \tau^n$ are uniformly bounded (as $\chi$ and $\varphi$ are in $L^\infty$). Thus by the Azuma-Hoeffding inequality for every $\be>0$ we have,
$$
\bbE_\mu[e^{\la\sum_{j=0}^{n-1}u\circ\tau^j}]\leq e^{\be^2n\|u\|_{L^\infty}^2}.
$$
Now the proof proceeds by using the Chernoff bounding method:
 by the Markov inequality for all $t>0$ we have
$$
\mu\left\{\sum_{j=0}^{n-1}u\circ\tau^j\geq tn\right\}\leq e^{-\be tn}e^{\be^2n\|u\|_{L^\infty}^2}.
$$
Taking $\be=\be_t=\frac{t}{2\|u\|_{L^\infty}}$ and replacing $u$ with $-u$ we get that
$$
\mu\left\{\pm\sum_{j=0}^{n-1}u\circ\tau^j\geq tn\right\}\leq e^{-\frac{nt^2}{4\|u\|_{L^\infty}}}.
$$
The proof of  Theorem \ref{Mart} (i) is completed now by noticing that 
\begin{equation}\label{Not}
\left\|S_n\varphi-\sum_{j=0}^{n-1}u\circ\tau^j\right\|_{L^\infty}=\left\|\chi-\chi\circ\tau^n\right\|_{L
^\infty}\leq 2\|\chi\|_{L^\infty}.
\end{equation}

Next, the proof of Theorem \ref{Mart} (ii) is completed by applying \cite[Proposition 7]{MPU} with the reverse martingale $(u\circ\tau^n)$ and using \eqref{Not}.

In order to prove Theorem \ref{ASIP Mart}, we apply  \cite[Theorem 3.2]{CM} with the bounded function $\varphi$ and the probability preserving system $(\Om\times X,\cF_0,\mu,\tau)$, whose transfer operator is $\cK$. Now, since we have assumed that $\mu(\varphi)=0$, in order for the conditions of \cite[Theorem 3.2]{CM} to be in force we need the following three estimates to hold
$$
\sum_{n\geq 2}n^{5/2}(\log n)^3\|\cK\varphi\|_{L^4(\mu)}^4<\infty\,\text{ and }\,\sum_{n\geq 2}n(\log n)^3\|\cK\varphi\|_{L^2(\mu)}^2<\infty
$$
and 
$$
\sum_{n\geq 2}\frac{(\log n)^3}{n^2}\left(\sum_{i=1}^{n}\sum_{j=0}^{n-i}\left\|\cK^i\big(\varphi \cK^j(\varphi)\big)-\mu\big(\varphi \cK^j(\varphi)\big)\right\|_{L^2(\mu)}\right)^2<\infty.
$$
The above three conditions are verified
by Proposition \ref{MainL} and the mixing rates specified in the formulation of Theorem \ref{ASIP Mart}, and the proof of Theorem \ref{ASIP Mart} is complete.
\end{proof}

\begin{proof}[Proof of Proposition \ref{MainL}]
(i) Since $L^\infty(\mu)$ is the dual of $L^1(\mu)$  and $\varphi$ and $\cK^n\varphi$ are $\cF_0$-measurable, it is enough to show that for every $g\in L^1(\Om\times X,\cF_0,\mu)$ so that $\|g\|_{L^1}\leq 1$ we have 
$$
\left|\int g\cdot(\cK^n\varphi)d\mu\right|\leq\gamma_n\|g\|_{L^1(\mu)}
$$
where $\gamma_n$ is one of the desired upper bounds.
To achieve that let us first note that $\cK^n$ is the dual of the restriction of the Koopman operator $f\to f\circ\tau^n$ acting on $\cF_0$-measurable functions. Thus, 
\begin{equation}\label{CrucialRel}
\int g\cdot(\cK^n\varphi)d\mu=\int \varphi\cdot (g\circ\tau^n) d\mu=
\int\left(\int\varphi_{\om}\cdot (g_{\sigma^n\om }\circ T_\om^n)\, d\mu_\om\right)d\bbP(\om)
\end{equation}
$$
=\int\left(\int (L_\om^n\varphi_{\om})\cdot g_{\sig^n\om}\, d\mu_{\sigma^n\om}\right)d\bbP(\om). 
$$

Now, using \eqref{Exp2} and that $\|\varphi\|_{K}=\esssup_{\om\in\Om}(K(\om)\|\varphi_\om\|_{BV})<\infty$  we get that 
$$
\left\|L_\om^n\varphi_{\om}-\mu_\om(\varphi_\om)\right\|_{L^\infty}\leq C_0\|\varphi\|_{K}e^{-\la n}.
$$
Hence, using also the $\sig$-invariance of $\bbP$,
$$
\int g\cdot(\cK^n\varphi)d\mu=\int \mu_\om(\varphi_{\om})\mu_{\sigma^n\om}(g_{\sigma^n\om})d\bbP(\om)+I
$$
where $|I|\leq Ce^{-\la n}\|g\|_{L^1(\mu)}$. 
Next, let us write 
$$
\mu_{\sigma^n\om}(g_{\sigma^n\om})=m(g_{\sig^n\om}h_{\sigma^n\om}).
$$
By \eqref{Exp1} we have 
$$
\left\|h_{\sigma^n\om}-\cL_{\sigma^{[n/2]}\om}^{n-[n/2]}\textbf{1}\right\|_{L^\infty}\leq C_0K(\sigma^{[n/2]}\om)e^{-\la n/2}\leq  C_0K(\om)e^{-(\la-\ve)n/2}.
$$
Observe next that since $\|1/h_\om\|_{BV}\leq K(\om)$ we have
$$m(|g|)=\mu_{\sig^n\om}(|g|/h_{\sigma^n\om})\leq C_0K(\sigma^n\om)\mu_\om(|g|)$$ 
 for every function $g$ and recall that $K(\sigma^n\om)\leq K(\om)e^{\ve n}$. Combining this with the previous estimates we get that
\begin{equation}\label{G est}
\left|m(g_{\sigma^n\om}h_{\sigma^n\om})-m(g_{\sig^n\om}\cL_{\sigma^{[n/2}]\om}^{n-[n/2]}\textbf{1})\right|C_0\leq K(\om)e^{-(\la-\ve)n/2}m(|g_{\sig^n\om}|)
\end{equation}
$$
\leq CK(\om)^2\mu_{\sigma^n\om}(|g_{\sigma^n\om}|)e^{-(\la-3\ve)n/2}.
$$
Therefore,
\begin{equation}\label{IJ}
\int g\cdot(\cK^n\varphi)d\mu=\int \mu_\om(\varphi_{\om})m(g_{\sig^n\om}\cL_{\sigma^{[n/2}]\om}^{n-[n/2]}\textbf{1})d\bbP(\om)+I+J
\end{equation}
where  $|I|\leq Ce^{-\la n}\|g\|_{L^1(\mu)}$ and $|J|\leq C'e^{-(\la-3\ve)n/2}\|g\|_{L^1(\mu)}$ and  we have used that $K(\om)^2\|\varphi_\om\|_{BV}$ is bounded.

Next, using \eqref{Exp1} and that $K(\om)$ is tempered we have $h_\om=\lim_{n\to\infty}\cL_{\sigma^{-n}\om}^n\textbf{1}$, and therefore $h_\om$ depends only on the coordinates $\om_j$ for $j\leq 0$. Thus  
$$
\mu_\om(\varphi_{\om})=F(\om_j; j\leq 0)
$$
for some measurable function $F$ so that $|F|\leq \|\varphi\|_{L^1(\mu)}$. Observe also that the random variable 
$$
G_n(\om)=m(g_{\sig^n\om}\cL_{\sigma^{[n/2}]\om}^{n-[n/2]}\textbf{1})
$$
depends only on $\om_j, j\geq [n/2]$ since $g_\om(x)$ is a function of $x$ and $\om_j, j\geq0$ (i.e. it factors through $\pi_0$).
 In the case when $h_\om\geq c^{-1}>0$ for some constant $c>0$ we have
$$
|G_n(\om)|=\left|\mu_{\sig^n\om}\big(g_{\sig^n\om}L^{n-[n/2]}_{\sig^{[n/2]}\om}(1/h_{\sig^{[n/2]}\om})\big)\right|\leq c\mu_{\sigma^n\om}(|g_{\sig^n\om}|).
$$
Thus, using also lemma \ref{PhiLem} we see that there is a constant $C>0$ so that
$$
\left|\int \mu_\om(\varphi_{\om})m(g_{\sig^n\om}\cL_{\sigma^{[n/2}]\om}^{n-[n/2]}\textbf{1})d\bbP(\om)\right|\leq C\phi_{[n/2],R}\int|G_n(\om)|d\bbP(\om)\leq cC\phi_{[n/2],R}\|g\|_{L^1(\mu)}
$$
where we have taken into account that $\int\mu_\om(\varphi_\om)d\bbP(\om)=\mu(\varphi)=0$. This, together with \eqref{IJ} and the previous estimates on $I$ and $J$, proves \eqref{1st}.

To prove \eqref{2nd}, we first use \eqref{G est}  in order to obtain that 
\begin{equation}\label{G est 1}
|G_n(\om)|\leq C\mu_{\sig^n\om}(|g_{\sigma^n\om}|)\big(1+CK^2(\om)e^{-(\la-3\ve)n/2}\big)\leq C'\mu_{\sig^n\om}(|g_{\sigma^n\om}|)K(\om)^2.
\end{equation}
Taking into account that
 $$
 \text{esssup}_{\om\in\Om}(\|\varphi_\om\|_{L^\infty}K(\om)^2)\leq C \text{esssup}_{\om\in\Om}(\|\varphi_\om\|_{BV}K(\om)^2)<\infty
 $$ 
 we conclude that
$G_n(\om)\mu_{\om}(\varphi_\om)$ is integrable. Now, we would like to apply Lemma \ref{PsiLem}, but the problem is that $G_n$ is not bounded. To overcome that,
 for each $M>0$ set $G_n^{(M)}(\om)=G_n(\om)\bbI(|G_n(\om)|\leq M)$. Then, since $G_n(\om)\mu_{\om}(\varphi_\om)$ is integrable, by the dominated convergence theorem we have 
$$
\int \mu_\om(\varphi_\om)G_n(\om)d\bbP(\om)=\lim_{M\to\infty}\int \mu_\om(\varphi_\om)G_n^{(M)}(\om)d\bbP(\om).
$$
Now, taking $n$ so that $\psi_{[n/2]}\leq 1/2$ and using that $\mu(\varphi)=0$ we get from Lemma \ref{PsiLem} that 
$$
\left|\int \mu_\om(\varphi_\om)G_n^{(M)}(\om)d\bbP(\om)\right|\leq2\left(\int |G_n^{(M)}(\om)\mu_\om(\varphi_\om)|d\bbP(\om)\right)\psi_{[n/2]}\leq
$$ 
$$
2\left(\int |G_n(\om)\mu_\om(\varphi_\om)|d\bbP(\om)\right)\psi_{[n/2]}.
$$
Using also \eqref{G est 1} and that $\text{esssup}_{\om\in\Om}(\|\varphi_\om\|_{BV}K(\om)^2)<\infty$, we conclude that
$$
\left|\int \mu_\om(\varphi_\om)G_n(\om)d\bbP(\om)\right|\leq 2\big(\esssup_{\om\in\Om}(K(\om)^2\|\varphi_\om\|_{BV})\big)C'\|g\|_{L^1}\psi_{[n/2]}
$$
and \eqref{2nd} follows (using also \eqref{IJ}).
\vskip0.1cm
(ii) First, since $\cK$ weakly contracts the $L^\infty$ norm (being defined through conditional expectations) and $\varphi$ is bounded we have
$$
\left\|\cK^i(\varphi \cK^j\varphi)-\mu\big(\cK^i(\varphi \cK^j\varphi)\big)\right\|_{L^\infty}\leq 2\|\varphi\|_{L^\infty}\|\cK^j\varphi\|_{L^\infty}.
$$
This together with Proposition \ref{MainL} (i) provides the desired estimate when $j\geq i$. The estimate in the case $i>j$ is carried out similarly to the proof of (i). Let $g\in L^1(\Om\times X,\mu,\cF_0)$.  Let us first show that 
\begin{equation}\label{First}
\int \cK^i(\varphi \cK^j\varphi)g\,d\mu=\int \mu_{\om}\big(\varphi_\om\cdot (\varphi_{\sigma^j\om}\circ T_\om^j)\big)\mu_{\sig^{i+j}\om}(g_{\sigma^{i+j}\om})d\bbP(\om)+I
\end{equation}
where $|I|\leq C_2e^{-\la i}$, and $C_2$ is some constant.

In order to prove \eqref{First}, using that $\cK$ satisfies the duality relation and the disintegration $\mu=\int \mu_\om d\bbP(\om)$ we first have
$$
\int \cK^i(\varphi \cK^j\varphi)g\,d\mu=\int (\varphi \cK^j\varphi)\cdot g\circ\tau^i\,d\mu=
\int \cK^j\varphi\cdot\big(\varphi\cdot (g\circ\tau^i)\big)d\mu=\int
\big(\varphi\cdot(\varphi\circ\tau^j)\big)\cdot g\circ\tau^{i+j}\,d\mu
$$
\begin{equation}\label{Eq}
=\int\left (\int \varphi_\om\cdot (\varphi_{\sigma^j\om}\circ T_\om^{j})\cdot (g_{\sigma^{i+j}\om}\circ T_\om^{i+j})d\mu_\om\right)d\bbP(\om)
\end{equation}
$$
=\int\left(\int L^{i+j}_\om\big(\varphi_\om \cdot (\varphi_{\sigma^j\om}\circ T_\om^j)\big)g_{\sigma^{i+j}\om}d\mu_{\sigma^{i+j}\om}\right)d\bbP(\om).
$$
Next, since $L_\om^n(f\circ T_\om^n)=f$ for every function $f$ and $n$, we have
$$
L^{i+j}_\om\big(\varphi_\om\cdot (\varphi_{\sigma^j\om}\circ T_\om^j)\big)=L_{\sigma^j\om}^i(\varphi_{\sigma^j\om}L_{\om}^j\varphi_\om).
$$
By \eqref{Exp2} we have 
$$
\left\|L_{\om}^j\varphi_\om-\mu_\om(\varphi_\om)\right\|_{BV}\leq K(\om)\|\varphi_\om\|_{BV}e^{-\la j}.
$$
In particular, 
$$
\|L_{\om}^j\varphi_\om\|_{BV}\leq CK(\om)\|\varphi_\om\|_{BV}
$$
for some constant $C$.
Since $\|uv\|_{BV}\leq C_0\|u\|_{BV}\|v\|_{BV}$  for every two functions $u,v$ we have 
$$
\|\varphi_{\sigma^j\om}L_{\om}^j\varphi_\om\|_{BV}\leq C_0CK(\om)\|\varphi_\om\|_{BV}\|\varphi_{\sigma^j\om}\|_{BV}.
$$
Thus by \eqref{Exp2},
$$
\left\|L_{\sigma^j\om}^i(\varphi_{\sigma^j\om}L_{\om}^j\varphi_\om)-\mu_{\sigma^j\om}(\varphi_{\sigma^j\om}L_{\om}^j\varphi_\om)\right\|_{BV}
$$
$$
\leq  C_0CK(\om)K(\sigma^j\om)\|\varphi_\om\|_{BV}\|\varphi_{\sigma^j\om}\|_{BV} e^{-\la i}\leq C_0C\|\varphi\|_{K}^2e^{-\la i},
$$
where $\|\varphi\|_K=\esssup_{\om\in\Om}(K(\om)\|\varphi_\om\|_{BV})$. Observe next that 
$$
\mu_{\sigma^j\om}(\varphi_{\sigma^j\om}L_{\om}^j\varphi_\om)=\mu_\om\big(\varphi_\om\cdot (\varphi_{\sigma^j\om}\circ T_\om^j)\big).
$$
The desired inequality \eqref{First} follows from the above estimates.

Observe that the function $\mu_{\om}(\varphi_\om\cdot  \varphi_{\sigma^j\om}\circ T_\om^j)$ depends only on $\om_k$ for $k\leq j$ and that it is bounded by $CK^{-2}(\om)$ for some constant $C>0$ (since $\esssup_{\om\in\Om}(K(\om)^2\|\varphi_\om\|_{BV})<\infty$). Therefore, the same arguments in the proof of (i) yield that 
$$
\int \mu_{\om}\big(\varphi_\om\cdot (\varphi_{\sigma^j\om}\circ T_\om^j)\big)\mu_{\sig^{i+j}\om}(g_{\sigma^{i+j}\om})d\bbP(\om)=
\int \mu_{\om}\big(\varphi_\om\cdot (\varphi_{\sigma^j\om}\circ T_\om^j)\big)d\bbP(\om)\cdot \int \mu_{\sig^{i+j}\om}(g_{\sigma^{i+j}\om})d\bbP(\om)+J
$$
where $|J|\leq \gamma_i\|g\|_{L^1}$ and
$\gamma_i$ is one of the right hand sides on the upper bounds in (i) (depending on the case) with $n$ replaced by $i$. Notice  next that 
$$
\int \mu_{\om}\big(\varphi_\om\cdot (\varphi_{\sigma^j\om}\circ T_\om^j)\big)d\bbP(\om)=\int \cK^i(\varphi \cK^j\varphi)d\mu
$$
(this can be seen by taking $g=1$ in \eqref{Eq}). Hence, 
$$
\left|\int\left(\cK^i(\varphi \cK^j\varphi)-\mu(\cK^i(\varphi \cK^j\varphi)\right)g\,d\mu\right|\leq C(e^{-\la i}+\gamma_i)\|g\|_{L^1}
$$
and the desired estimate follows again since $L^\infty$ is the dual of $L^1$.
\end{proof}
\section{A vector valued ASIP for skew products with uniformly expanding fiber maps and exponentially fast $\al$-mixing base maps}
Let us first explain why the matrix $\Sigma^2$ exists. For a fixed vector $v$ the limit $s_v^2=\lim_{n\to\infty}\frac1n\bbE[(S_n\cdot v)^2]$ 
exists, by considering the real-valued  observable $\varphi\cdot v$. Then the matrix $\Sigma^2$ from Theorem \ref{ASIP GOU} is given by $(\Sig^2)_{i,j}=\frac{1}2\big(s_{e_i+e_j}^2-s_{e_i}^2-s_{e_j}^2\big)$. 
This matrix satisfies $\Sigma^2 v\cdot v=s_v^2$ and so it is not positive definite if and only if $\varphi\cdot v$ is a coboundary for some unit vector $v$. Note that this part does not require $T_\om$ to be uniformly expanding.

We assume next that there exist constants $C>0$ and $\delta\in(0,1)$ so that for $\mathbb P$ a.e. $\omega$ we have  
\begin{equation}\label{App1}
\|\mathcal L_\omega^n \textbf{1}-h_{\sigma^n\omega}\|_{BV}\leq C\delta^n
\end{equation}
(this is the uniform expansion assumption).


The proof of Theorem \ref{ASIP GOU} relies on an application of  \cite[Theorem 1.2]{GO}.
The main condition of \cite[Theorem 1.2]{GO} is the content of the following lemma. Once the lemma is proven Theorem \ref{ASIP GOU} follows from \cite[Theorem 1.2]{GO} applied with an arbitrary large $p$.
\begin{lemma}
There exists $\ve_0>0$, $c,C>0$ such that for any $n,m>0$, $b_1<b_2<...<b_{n+m+1}$,  $k>0$ and $t_1,...,t_{n+m}\in \bbR^d$ with $|t_j|\leq \ve_0$ we have
\begin{equation}\label{EQQ}
			\begin{split}
				\Big|\mathbb E_\mu&\big(e^{i\sum_{j=1}^nt_j \cdot (\sum_{\ell=b_j}^{b_{j+1}-1}B_\ell)+i\sum_{j=n+1}^{n+m}t_j \cdot (\sum_{\ell=b_j+k}^{b_{j+1}+k-1}B_\ell)}\big)\\
				&-\mathbb E_\mu\big(e^{i\sum_{j=1}^nt_j \cdot (\sum_{\ell=b_j}^{b_{j+1}-1}B_\ell)}\big)\cdot\mathbb E_\mu\big(e^{i\sum_{j=n+1}^{n+m}t_j \cdot (\sum_{\ell=b_j+k}^{b_{j+1}+k-1}B_\ell)}\big)\Big|\\
				&\leq C^{n+m} e^{-ck},
			\end{split}
		\end{equation}
		where $B_\ell=\varphi\circ\tau^\ell$.
\end{lemma}

\begin{proof}
First, denoting by $\bbE_\om$ the expectation with respect to $\mu_\om$, by \cite[Lemma 24]{DHS} there are $\ve_0>0$, $c,C>0$  with the property that  for every $n,m>0$, $b_1<b_2<...<b_{n+m+1}$,  $k>0$ and $t_1,...,t_{n+m}\in \bbR^d$ such that $|t_j|\leq \ve_0$,
\begin{equation}\label{EQQ fib}
			\begin{split}
				\Big|\mathbb E_\omega&\big(e^{i\sum_{j=1}^nt_j \cdot (\sum_{\ell=b_j}^{b_{j+1}-1}A_\ell)+i\sum_{j=n+1}^{n+m}t_j \cdot (\sum_{\ell=b_j+k}^{b_{j+1}+k-1}A_\ell)}\big)\\
				&-\mathbb E_\omega\big(e^{i\sum_{j=1}^nt_j \cdot (\sum_{\ell=b_j}^{b_{j+1}-1}A_\ell)}\big)\cdot\mathbb E_\omega\big(e^{i\sum_{j=n+1}^{n+m}t_j \cdot (\sum_{\ell=b_j+k}^{b_{j+1}+k-1}A_\ell)}\big)\Big|\\
				&\leq C^{n+m} e^{-ck},
			\end{split}
		\end{equation}
		where $\mathbb E_\omega(g)=\int g h_\omega dm$ and 
		\[A_\ell:=\varphi_{\sigma^\ell \omega} \circ T_\omega^\ell, \quad \ell\in \mathbb N. \]
		
Let 
$$
G(\omega)=\mathbb E_\omega\big(e^{i\sum_{j=1}^nt_j \cdot (\sum_{\ell=b_j}^{b_{j+1}-1}A_\ell)}\big)
$$
and 
$$
F(\omega)=\mathbb E_\omega\big(e^{i\sum_{j=n+1}^{n+m}t_j \cdot (\sum_{\ell=b_j+k}^{b_{j+1}+k-1}A_\ell)}\big).
$$
Then with $B_\ell=\varphi\circ\tau^\ell$ we have
\begin{equation}\label{EQQ1}
			\begin{split}
				\Big|\mathbb E_\mu&\big(e^{i\sum_{j=1}^nt_j \cdot (\sum_{\ell=b_j}^{b_{j+1}-1}B_\ell)+i\sum_{j=n+1}^{n+m}t_j \cdot (\sum_{\ell=b_j+k}^{B_{j+1}+k-1}B_\ell)}\big)\\
				&-\mathbb E_\mu\big(e^{i\sum_{j=1}^nt_j \cdot (\sum_{\ell=b_j}^{b_{j+1}-1}B_\ell)}\big)\cdot\mathbb E_\mu\big(e^{i\sum_{j=n+1}^{n+m}t_j \cdot (\sum_{\ell=b_j+k}^{b_{j+1}+k-1}B_\ell)}\big)\Big|\\
				&\leq C^{n+m} e^{-ck}+|\text{Cov}_{\mathbb P}(G,F)|.
			\end{split}
		\end{equation}
		Using \eqref{App1} and that $(T_\om)_*\mu_\om=\mu_{\sig\om}$ we get that there are $k_0\in\mathbb Z $ and functions $G_1$ and $F_1$ so that
	$$
	\|G(\omega)-G_1(...,\omega_{k_0-1},\omega_{k_0+[k/4]})\|_{L^\infty}\leq C'\delta^{k/4}
	$$
	and 
	$$
	\|G(\omega)-G_1(\omega_{k_0+k-[k/4]},\omega_{k_0+k-[k/4]+1},...)\|_{L^\infty}\leq C'\delta^{k/4}.
	$$
	Thus,
	$$
	|\text{Cov}_{\mathbb P}(G,F)|\leq |\text{Cov}_{\mathbb P}(G_1,F_1)|+C''\delta^{k/4}
	$$
	where we have used that $G_1,G_2,G$ and $F$ are uniformly bounded (so the above constants $C',C''$ do not depend on the choice of $b_j,t_j$ etc.). On the other hand, by \eqref{alpha}, 
	$$
	|\text{Cov}_{\mathbb P}(G_1,F_1)|\leq C'''\alpha^{k/2}.
	$$
	Thus,
	\begin{equation}\label{EQQ2}
			\begin{split}
				\Big|\mathbb E_\mu&\big(e^{i\sum_{j=1}^nt_j \cdot (\sum_{\ell=b_j}^{b_{j+1}-1}B_\ell)+i\sum_{j=n+1}^{n+m}t_j \cdot (\sum_{\ell=b_j+k}^{B_{j+1}+k-1}B_\ell)}\big)\\
				&-\mathbb E_\mu\big(e^{i\sum_{j=1}^nt_j \cdot (\sum_{\ell=b_j}^{b_{j+1}-1}B_\ell)}\big)\cdot\mathbb E_\mu\big(e^{i\sum_{j=n+1}^{n+m}t_j \cdot (\sum_{\ell=b_j+k}^{b_{j+1}+k-1}B_\ell)}\big)\Big|\\
				&\leq C^{n+m} e^{-ck}+C''\delta^{\delta k/4}+C'''\alpha^{k/2}.
			\end{split}
	\end{equation}
\end{proof}

\section{Extensions and generalizations and additional results, a short discussion}
In this section we will describe a few additional results which can also be obtained using the methods of the current paper.
In order not to overload the paper the section is presented in a form of a discussion rather than explicit formulations of theorems.
\subsection{More general mixing base maps for continuous in $\om$ transfer operators, a short discussion}
Let $(\xi_n)_{n\in\mathbb Z}$ be a stationary process taking values on a metric space $(\mathcal Y,d)$ satisfying the following approximation and mixing conditions:

There are sub-$\sigma$-algebras $\cG_{n,m}$ on the underlying probability space so that $\cG_{n,m}\subset \cG_{n_1,m_1}$ if $[n,m]\subset [n_1,m_1]$ and for each r and n there is an $\cG_{n-r,n+r}$ measurable random variable $\xi_{n,r}$ so that:
\vskip0.1cm

(1)\textbf{ approximation}: $\|d(\xi_n,\xi_{n,r})\|_{L^\infty}\leq A_1\beta^r$, $\beta\in(0,1)$
\vskip0.1cm
(2)\textbf{ mixing}: the sequences $(\xi_{2nr,r})_{n\in\mathbb Z}$ are  $\alpha$ (or $\phi_R$ or $\psi$) mixing uniformly in $r$.
\vskip0.1cm

We note that the above uniform approximation by $\al$-mixing sequences applies to Young towers, when $\al_n=O(n^{-(p-2)})$ is the tails of the tower are $O(n^{-p})$ for some $p\geq 3$. We can also take several classes of smooth maps on the interval or Gibbs-Markov maps \cite{Jon} for which such an approximation holds with $\psi_n=O(\del^n)$ for some $\del\in(0,1)$.

Let $(\Omega,\mathcal F,\mathbb P,\sigma)$ be the shift system constructed as before. Then  all the results stated in the paper hold true when  $\omega\to\mathcal L_{\omega}$ and $\omega\to\varphi_\omega$ are H\"older continuous in $\omega$ (on a set with probability $1$). The main point is that Lemma \ref{LemmmaAp} and the similar approximations used in the construction of the martingale (i.e. in the  proof of Proposition \ref{MainL}) can be obtained by first approximating (taking $r=r_n=\ve_0n$ for some small $\ve_0$) and using the mixing conditions on the approximating sequences. The main reason we did not include such results in the body of the paper is that it would make the notations more complicated, and that the additional essentially global regularity assumptions on the transfer operators are somehow less natural.

\subsection{Extension to random Gibbs measures}
Let us consider now the random expanding maps $T_\om$ as in \cite{MSU}. 
Let $\mu_\om=h_\om \nu_\om$ be a random Gibbs measure corresponding to a given random logarithmically $\al$-H\"older continuous potential, and let $\la_\om$ be the exponent of the random pressure. Namely, if $\cL_\om$ is the transfer operator corresponding to the random potential, then 
$$
\cL_\om h_\om=\la_\om h_{\sig\om}, (\cL_\om)^*\nu_{\sigma\om}=\la_\om\nu_{\om}.
$$
Next, for the sake of simplicity let us consider here random expanding maps as in \cite[Chapter 5]{HK}.
Then there is a constant $K>0$ so that 
 with $\tilde \cL_\om=\cL_\om/\la_\om$ we have  
$$
\|\tilde\cL_\om^n-\nu_\om\otimes h_{\sig^n\om}\|_{Holder}\leq Ke^{-\la n}
$$
where $\|\cdot\|_{Holder}$ is the usual H\"older norm corresponding to the exponent $\al$ and $\nu\otimes h(g)=\nu(g)h$.  Plugging in $g=\textbf{1}$ we get similar estimates to the ones we had in \eqref{Exp1}:
$$
\|\tilde\cL_\om^n- h_{\sig^n\om}\|_{Holder}\leq Ke^{-\la n}.
$$
Remark also that $h_\om\geq c>0$ for some constant $c>0$ (see \cite{HK}).

The main additional difficulty here is to estimate expressions of the form $\mu_\om(F_\om)$ (as in Lemma \ref{LemmmaAp}) by functions of the coordinates in places $j$ for $|j|\leq n$. Once this is achieved, we can use the approximation argument (similarly to Lemma \ref{LemmmaAp}) which was essential in the proofs of all of the results stated in the body of the paper. The main difference in comparison with the case when $\nu_\om=m$ does not depend on $\om$ is that now we need to approximate $\nu_\om$ by functions of the first $n$ coordinates (exponentially fast in $n$). For uniformly expanding maps, this follows from the construction of $\nu_\om$ as a certain uniform limit (see \cite[Ch. 4-5]{HK}).

\subsection{Extension to nonconventional sums (multiple recurrences)}
Let us consider partial ``nonconventional" sums of the form 
$$
S_n\varphi=\sum_{m=1}^{n}\prod_{j=1}^\ell \varphi\circ\tau^{q_j(m)}
$$
where $\ell$ is an integer and $q_j(n)$ are positive integer-valued sequences. The statistical properties of such sums were  studies for several classes of expanding or hyperbolic maps (in particular), see \cite{KifPTRF, KV, AIHP} and references therein.
When all $q_j$'s are polynomials, we believe that all the results obtained using the method of cumulants (i.e. Theorems \ref{CLT}, \ref{Thm:ModDevNonc}, \ref{Thm2.4}, \ref{MomThm} and an appropriate version of 
Theorem \ref{FCLT}) can be obtained for such sums exactly as in \cite{AIHP}, relying on a version of Proposition \ref{PropMulti} applied with $\rho(n,m)=\max_{1\leq i,j\leq\ell}|q_i(m)-q_j(n)|$. The main idea is that by induction on the number of blocks we can show that the conditions of Proposition \ref{GorcCor} with that $\rho=\rho_\ell$ hold true for 
$$X_m=\prod_{j=1}^\ell \varphi\circ\tau^{q_j(m)}.$$
That is, by an inductive argument similar to the one in \cite[Corollary 1.3.11]{HK}, we can prove the following result.

\begin{lemma}
Let $r\in\bbN$ and 
let $B_1,B_2,...,B_k$ be finite subsets of $\bbN$ so that the distance between $B_j$ and $B_{j+1}$ is $d_j$. Set $r_j=[d_j/3]$. Let $\cC=\{\cC_j:\,1\leq j\leq s\}$ be a partition of $\{1,2,...,k\}$ and set $Y_j=\prod_{k\in \cC_j}\prod_{u\in B_k}\varphi\circ\tau^{u}$. Then, assuming that $\|\varphi\|_{L^\infty}\leq 1$ and that $\esssup_{\om\in\Om}(K(\om)\|\varphi_\om\|_{BV})\leq1$, there is an absolute constant $A>1$ so that 
$$
\left|\bbE_\mu\left[\prod_{j=1}^s Y_j\right]-\prod_{j=1}^s\bbE_\mu[Y_j]\right|\leq A^m\sum_{j=1}^{m}(\delta^{r_j}+\alpha([r_j]))
$$
where $\del=e^{-(\la-3\ve)/2}\in(0,1)$.
\end{lemma}
We note that in order to prove a version of the functional CLT for the sums above we first need to use the arguments in \cite{KV,HK2} to compute the variance of the limiting Gaussian, which for general polynomials might differ from a Brownian motion, and this can also be done by using the above lemma.



\end{document}